%% file: gsf_calc_var_final_2.tex
\documentclass[a4paper,reqno]{amsart}
\usepackage[latin9]{inputenc}
\usepackage{xcolor}
\usepackage{amstext}
\usepackage{amsthm}
\usepackage{amssymb}
\usepackage{graphicx}
\usepackage{esint}
\PassOptionsToPackage{normalem}{ulem}
\usepackage{ulem}
\usepackage[unicode=true,pdfusetitle,
 bookmarks=true,bookmarksnumbered=false,bookmarksopen=false,
 breaklinks=false,pdfborder={0 0 1},backref=false,colorlinks=true]
 {hyperref}
\usepackage{breakurl}

\makeatletter

\providecolor{lyxadded}{rgb}{0,0,1}
\providecolor{lyxdeleted}{rgb}{1,0,0}

\numberwithin{equation}{section}
\numberwithin{figure}{section}
\usepackage{enumitem}		
\newlength{\lyxlistindent}      
\theoremstyle{plain}
\newtheorem{thm}{\protect\theoremname}
  \theoremstyle{definition}
  \newtheorem{defn}[thm]{\protect\definitionname}
  \theoremstyle{plain}
  \newtheorem{lem}[thm]{\protect\lemmaname}
  \theoremstyle{remark}
  \newtheorem{rem}[thm]{\protect\remarkname}
  \theoremstyle{plain}
  \newtheorem{cor}[thm]{\protect\corollaryname}
  \theoremstyle{plain}
  \newtheorem{prop}[thm]{\protect\propositionname}

\usepackage{nicefrac}

\setlist[enumerate]{leftmargin=*,label=(\roman*),align=left}

\def\1{\mathbb 1}

\newcommand{\rc}{\rcrho}

\newcommand{\pdd}[1]{\frac{\partial}{\partial{#1}}}

\newcommand{\pdddp}[1]{\frac{\partial^2 }{\partial{#1}^2}}

\newcommand{\dd}{\mathrm{d}}

\newcommand{\ve}{\varepsilon}
\newcommand{\der}[1]{\frac{\dd}{\dd #1}}
\newcommand{\difff}[1]{\frac{\dd^2}{\dd #1^2}}

\newcommand{\gsf}{{}^{\scriptscriptstyle \rho}\mathcal{GC}^\infty}

\newcommand{\ra}{\longrightarrow}

\newcommand{\field}[1]{\mathbb{#1}}
\newcommand{\R}{\field{R}} 
\newcommand{\N}{\field{N}} 

\newcommand{\D}{\mathcal{D}}

\newcommand{\eps}{\varepsilon} 
\renewcommand{\phi}{\varphi}
\newcommand{\diff}[1]{\,\hbox{\rm d}#1} 

\newcommand{\Coo}{\mbox{\ensuremath{\mathcal{C}}}^{\infty}} 
\newcommand{\Gcinf}{\mathcal{G}\cinfty} 

\newcommand{\Rtil}{\widetilde \R} 
\newcommand{\gs}{\mathcal{G}^s} 
\newcommand{\ns}{\mathcal{N}^s} 

\newcommand{\sint}[1]{\langle#1\rangle} 
\newcommand{\Eball}{B^{{\scriptscriptstyle \text{\rm E}}}} 

\newcommand{\RC}[1]{{}^{\scriptscriptstyle #1}\Rtil}
\newcommand{\rcrho}{\RC{\rho}}
\newcommand{\csp}[1]{{\text{\rm c}}({#1})}
\newcommand{\fcmp}{\Subset_{\text{f}}}

\newcommand{\ptind}{\displaystyle \mathop {\ldots\ldots\,}} 
\newcommand{\then}{\quad \Longrightarrow \quad} 

\newcommand{\cinfty}{{\mathcal C}^\infty}

\newcommand{\comp}{\Subset}
\newcommand{\esm}{{\mathcal E}_M}

\newcommand{\Om}{\Omega}

\newcommand{\sse}{\subseteq}

\newcommand{\rti}{\RC{\rho}}

\newcommand{\no}[1]{| #1|}

\makeatother

  \providecommand{\corollaryname}{Corollary}
  \providecommand{\definitionname}{Definition}
  \providecommand{\lemmaname}{Lemma}
  \providecommand{\propositionname}{Proposition}
  \providecommand{\remarkname}{Remark}
\providecommand{\theoremname}{Theorem}

\begin{document}

\title[Calculus of variations for GF]{The classical theory of calculus of variations for generalized functions}

\author{Alexander Lecke \and Lorenzo Luperi Baglini \and Paolo Giordano}

\thanks{A.~Lecke has been supported by the uni:doc fellowship programme
of the University of Vienna}

\address{\textsc{University of Vienna, Austria}}

\email{alexander.lecke@univie.ac.at }

\thanks{L.~Luperi Baglini has been supported by grant M1876-N35 of the Austrian
Science Fund FWF}

\address{\textsc{University of Vienna, Austria}}

\email{lorenzo.luperi.baglini@univie.ac.at}

\thanks{P.\ Giordano has been supported by grants P25311-N25 and P25116-N25
of the Austrian Science Fund FWF}

\email{paolo.giordano@univie.ac.at }


\begin{abstract}
We present an extension of the classical theory of calculus of variations
to generalized functions. The framework is the category of generalized
smooth functions, which includes Schwartz distributions while sharing
many nonlinear properties with ordinary smooth functions.
We prove full connections between extremals and Euler-Lagrange equations,
classical necessary and sufficient conditions to have a minimizer,
the necessary Legendre condition, Jacobi's theorem on conjugate points
and Noether's theorem. We close with an application to low regularity Riemannian geometry.
\end{abstract}

\maketitle

\section{Introduction and motivations}

Singular problems in the calculus of variations have longly been studied
both in mathematics and in relevant applications (see, e.g., \cite{Gra30,Dav88,Tuc93,KKO:08}
and references therein). In this paper, we introduce an approach to
variational problems involving singularities that allows the extension
of the classical theory with very natural statements and proofs. We
are interested in extremizing functionals which are either distributional
themselves or whose set of extremals includes generalized functions.
Clearly, distribution theory, being a linear theory, has certain difficulties
when nonlinear problems are in play.

To overcome this type of problems, we are going to use the category
of generalized smooth functions, see \cite{GiKu16,GK15,Gi-Ku-St15,GKV}.
This theory seems to be a good candidate, since it is an extension
of classical distribution theory which allows to model nonlinear singular
problems, while at the same time sharing many nonlinear properties
with ordinary smooth functions, like the closure with respect to composition
and several non trivial classical theorems of the calculus. One could
describe generalized smooth functions as a methodological restoration
of Cauchy-Dirac's original conception of generalized function, see
\cite{Dir26,Lau89,Kat-Tal12}. In essence, the idea of Cauchy and
Dirac (but also of Poisson, Kirchhoff, Helmholtz, Kelvin and Heaviside)
was to view generalized functions as suitable types of smooth set-theoretical
maps obtained from ordinary smooth maps depending on suitable infinitesimal
or infinite parameters. For example, the density of a Cauchy-Lorentz
distribution with an infinitesimal scale parameter was used by Cauchy
to obtain classical properties which nowadays are attributed to the
Dirac delta, cf.~\cite{Kat-Tal12}.

In the present work, the foundation of the calculus of variations
is set for functionals defined by arbitrary generalized functions.
This in particular applies to any Schwartz distribution and any Colombeau
generalized function, and hence justifies the title of the present paper.

For example, during the last years, the study of low regularly Riemannian
and Lo\-ren\-tzian geometry was intensified and made a huge amount
of progress (cf.~\cite{KSS:14,KSSV:13,Milena,LY:06,LSS:13,SSLP:16}).
It was shown that the exponential map is a bi-Lipschitz homeomorphism
when metrics $g\in\mathcal{C}^{1,1}$ are considered, \cite{M:13,KSS:14},
or that Hawking's singularity theorem still holds when $g\in\mathcal{C}^{1,1}$,
see \cite{KSSV:13}. However, calculus of variations in the classical
sense may cease to hold when metrics with $\mathcal{C}^{1,1}$ regularity,
or below, are considered \cite{HE:73, PhDAlex}. This motivates the search 
for an alternative. In fact, if $p$, $q\in\R^{d}$ and $\Omega(p,q)$
denotes the set of all Lipschitz continuous curves connecting $p$
and $q$, the natural question about what curves $\gamma\in\Omega(p,q)$
realize the minimal $g$-length leads to the corresponding
geodesic equation, but the Jacobi equation is not rigorously
defined. To be more precise: The Riemannian curvature tensor exists
only as an $\text{L}_{\text{loc}}^{\infty}$ function on $\R^{d}$
and is evaluated along $\gamma$. However, the image $\text{Im}(\gamma)$
of $\gamma$ has Lebesgue-measure zero, if $d>1$. Thus we cannot
state the Jacobi equations properly.

In order to present a possible way out of the aforementioned problems,
the singular metric $g$ is embedded as a generalized smooth function.
In this way, the embedding $\iota(g)$ has derivatives of all orders,
valued in a suitable non-Archimedean ring\footnote{I.e.~a ring that contains infinitesimal and infinite numbers.}
$\rc\supseteq\R$, and behaves very closely to a standard smooth function.
We apply our extended calculus of variations to the generalized Riemannian
space $(\rc^{d},\iota(g))$, and sketch a way to translate the given
problem into the language of generalized smooth functions, solve it
there, and translate it back to the standard Riemannian space $(\R^{d},g)$.
Clearly, the process of embedding the singular metric $g$ using $\iota(g)$
introduces infinitesimal differences. This is typical in a non-Archimedean
setting, but the notion of \emph{standard part} comes to help: if
$x\in\rc$ is infinitely close to a standard real number $s$, i.e.~$|x-s|\le r$
for all $r\in\R_{>0}$, then the standard part of $x$ is exactly
$s$. We then show that (assuming that $(\R^{d},g)$ is geodesically
complete) the standard part of the minimal length in the sense of
generalized smooth functions is the minimal length in the classical
sense, and give a simple way to check if a given (classical) geodesic
is a minimizer of the length functional or not. In this way, the framework
of generalized smooth functions is presented as a method to solve
standard problems rather than a proposal to switch into a new setting.

The structure of the present paper is as follows. We start with an
introduction into the setting of generalized smooth functions and
give basic notions concerning generalized smooth functions and their
calculus that are needed for the calculus of variations (Section \ref{sec:Basic-notions}).
The paper is self-contained in the sense that it contains all the
statements required for the proofs of calculus of variations we are
going to present. If proofs of preliminaries are omitted, we clearly
give references to where they can be found. Therefore, to understand
this paper, only a basic knowledge of distribution theory is needed.\\
In Section \ref{sec:Preliminary-results}, we obtain some preliminary
lemmas regarding the calculus of variations with generalized smooth
functions. The first variation and the notion of critical point will
be defined and studied in section \ref{sec:First-variation}. We prove
the fundamental lemma of calculus of variations and the full connection
between critical points of a given functional and solutions of the
corresponding Euler-Lagrange equation. In section \ref{sec:second-variation},
we study the second variation and define the notion of local minimizer.
We also extend to generalized functions classical necessary and sufficient
conditions to have a minimizer, and we give a proof of the Legendre
condition. In Section \ref{sec:Jacobi-fields}, we introduce the notion
of Jacobi field and extend to generalized functions the definition
of conjugate points, so as to prove the corresponding Jacobi theorem.
In Section \ref{sec:Noether's-theorem}, we extend the classical Noether's
theorem. We close with an application to $\mathcal{C}^{1,1}$ Riemannian
geometry in Section \ref{sec:Application}.

Note that the work \cite{KKO:08} already established the calculus
of variations in the setting of Colombeau generalized functions by
using a comparable methodological approach. Indeed, generalized smooth
functions are related to Colombeau generalized functions, and one
could say that the former is a minimal extension of the latter so
as to get more general domains for generalized functions and hence
the closure with respect to composition and a better behaviour on
unbounded sets. However, there are some conceptual advantages in our
approach. 
\begin{enumerate}
\item Whereas generalized smooth functions are closed with respect to composition,
Colombeau generalized functions are not. This forces \cite{KKO:08}
to consider only functionals defined using compactly supported Colombeau
generalized functions, i.e.~functions assuming only finite values,
or tempered generalized function. 
\item The authors of \cite{KKO:08} are forced to consider the so called
compactly supported points $c(\Omega)$ (i.e.~finite points in $\Omega\subseteq\R^{n}$),
where the setting of generalized smooth functions gives the possibility
to consider more natural domains like the interval $[a,b]\subseteq\rcrho$.
This leads us to extend in a natural way the statements of classical
results of calculus of variations. Moreover, all our results still
hold when we take as $a$, $b\in\rc$ two infinite numbers such that
$a<b$, or as boundary points two unbounded points $p$, $q\in\rc^{d}$.
\item Furthermore, the theory of generalized smooth functions was developed
to be very user friendly, in the sense that one can avoid cumbersome
``$\varepsilon$-wise'' proofs quite often, whereas the proofs in
\cite{KKO:08} frequently use this technique. Thus, one could say
that some of the proofs based on generalized smooth functions are
more ``intrinsic'' and close to the classical proofs in a standard
smooth setting. This allows a smoother approach to this new framework. 
\item The setting of generalized smooth functions depends on a fixed
infinitesimal net $\left(\rho_{\eps}\right)_{\eps\in(0,1]}\downarrow0$,
whereas the Colombeau setting considers only $\rho_{\eps}=\eps$.
This added degree of freedom allows to solve singular differential
equations that are unsolvable in the classical Colombeau setting and
to prove a more general Jacobi theorem on conjugate points. 
\item In \cite{KKO:08} only the notion of global minimizer is defined,
whereas we define the notion of local minimizer as in \cite{GeFo00}
using a natural topology in space of generalized smooth curves.
\item We obtain more classical results like the Legendre condition, and
the classical results about Jacobi fields and conjugate points.
\item In addition, note that the Colombeau generalized functions can be
embedded into generalized smooth functions. Thus our approach is a
natural extension of \cite{KKO:08}.
\end{enumerate}

\section{\label{sec:Basic-notions}Basic notions}

\subsubsection*{The new ring of scalars}

In this work, $I$ denotes the interval $(0,1]\subseteq\R$ and we
will always use the variable $\eps$ for elements of $I$; we also
denote $\eps$-dependent nets $x\in\R^{I}$ simply by $(x_{\eps})$.
By $\N$ we denote the set of natural numbers, including zero.

We start by defining the new simple non-Archimedean ring of scalars
that extends the real field $\R$. The entire theory is constructive
to a high degree, e.g.~no ultrafilter or non-standard method is used.
For all the proofs of results in this section, see \cite{GKV,GiKu16,Gi-Ku-St15}. 
\begin{defn}
\label{def:RCGN}Let $\rho=(\rho_{\eps})\in\R^{I}$ be a net such
that $\lim_{\eps\to0}\rho_{\eps}=0^{+}$, then

\begin{enumerate}
\item $\mathcal{I}(\rho):=\left\{ (\rho_{\eps}^{-a})\mid a\in\R_{>0}\right\} $
is called the \emph{asymptotic gauge} generated by $\rho$. The net
$\rho$ is called a \emph{gauge}.
\item If $\mathcal{P}(\eps)$ is a property of $\eps\in I$, we use the
notation $\forall^{0}\eps:\,\mathcal{P}(\eps)$ to denote $\exists\eps_{0}\in I\,\forall\eps\in(0,\eps_{0}]:\,\mathcal{P}(\eps)$.
We can read $\forall^{0}\eps$ as \emph{for $\eps$ small}. 
\item We say that a net $(x_{\eps})\in\R^{I}$ \emph{is $\rho$-moderate},
and we write $(x_{\eps})\in\R_{\rho}$ if $\exists(J_{\eps})\in\mathcal{I}(\rho):\ x_{\eps}=O(J_{\eps})$
as $\eps\to0^{+}$. 
\item Let $(x_{\eps})$, $(y_{\eps})\in\R^{I}$, then we say that $(x_{\eps})\sim_{\rho}(y_{\eps})$
if $\forall(J_{\eps})\in\mathcal{I}(\rho):\ x_{\eps}=y_{\eps}+O(J_{\eps}^{-1})$
as $\eps\to0^{+}$. This is a congruence relation on the ring $\R_{{\scriptscriptstyle \rho}}$
of moderate nets with respect to pointwise operations, and we can
hence define 
\[
\RC{\rho}:=\R_{{\scriptscriptstyle \rho}}/\sim_{\rho},
\]
which we call \emph{Robinson-Colombeau ring of generalized numbers},
\cite{Rob73,Col92}. We denote the equivalence class $x\in\rti$ simply
by $x=:[x_{\eps}]:=[(x_{\eps})]_{\sim}\in\rti$. 
\end{enumerate}
\end{defn}
In the following, $\rho$ will always denote a net as in Def.~\ref{def:RCGN}.
The infinitesimal $\rho$ can be chosen depending on the class of
differential equations we need to solve for the generalized functions
we are going to introduce, see \cite{GiLu15}. For motivations concerning
the naturality of $\rti$, see \cite{Gi-Ku-St15}.

We can also define an order relation on $\RC{\rho}$ by saying that
$[x_{\eps}]\le[y_{\eps}]$ if there exists $(z_{\eps})\in\R^{I}$
such that $(z_{\eps})\sim_{\rho}0$ (we then say that $(z_{\eps})$
is \emph{$\rho$-negligible}) and $x_{\eps}\le y_{\eps}+z_{\eps}$
for $\eps$ small. Equivalently, we have that $x\le y$ if and only
if there exist representatives $[x_{\eps}]=x$ and $[y_{\eps}]=y$
such that $x_{\eps}\le y_{\eps}$ for all $\eps$. Clearly, $\RC{\rho}$
is a partially ordered ring. The usual real numbers $r\in\R$ are
embedded in $\RC{\rho}$ considering constant nets $[r]\in\RC{\rho}$.

Even if the order $\le$ is not total, we still have the possibility
to define the infimum $\min\left([x_{\eps}],[y_{\eps}]\right):=[\min(x_{\eps},y_{\eps})]$,
and analogously the supremum function $\max\left([x_{\eps}],[y_{\eps}]\right):=\left[\max(x_{\eps},y_{\eps})\right]$
and the absolute value $|[x_{\eps}]|:=[|x_{\eps}|]\in\RC{\rho}$.
Note, e.g., that $x\le z$ and $-x\le z$ imply $|x|\le z$. In the
following, we will also use the customary notation $\RC{\rho}^{*}$
for the set of invertible generalized numbers. Our notations for intervals
are: $[a,b]:=\{x\in\RC{\rho}\mid a\le x\le b\}$, $[a,b]_{\R}:=[a,b]\cap\R$,
and analogously for segments $[x,y]:=\left\{ x+r\cdot(y-x)\mid r\in[0,1]\right\} \subseteq\RC{\rho}^{n}$
and $[x,y]_{\R^{n}}=[x,y]\cap\R^{n}$. Finally, we write $x\approx y$
to denote that $|x-y|$ is an infinitesimal number, i.e.~$|x-y|\le r$
for all $r\in\R_{>0}$. This is equivalent to $\lim_{\eps\to0^{+}}|x_{\eps}-y_{\eps}|=0$
for all representatives $x=[x_{\eps}]$ and $y=[y_{\eps}]$.

\subsubsection*{Topologies on $\RC{\rho}^{n}$}

On the $\RC{\rho}$-module $\RC{\rho}^{n}$, we can consider the natural
extension of the Euclidean norm, i.e.~$|[x_{\eps}]|:=[|x_{\eps}|]\in\RC{\rho}$,
where $[x_{\eps}]\in\RC{\rho}^{n}$. Even if this generalized norm
takes values in $\RC{\rho}$, it shares several properties with usual
norms, like the triangular inequality or the property $|y\cdot x|=|y|\cdot|x|$.
It is therefore natural to consider on $\RC{\rho}^{n}$ topologies
generated by balls defined by this generalized norm and a set of radii
$\mathfrak{R}$: 
\begin{defn}
\label{def:setOfRadii}Let $\mathfrak{R}\in\left\{ \RC{\rho}_{\ge0}^{*},\R_{>0}\right\} $,
$c\in\RC{\rho}^{n}$ and $x$, $y\in\RC{\rho}$, then:

\begin{enumerate}
\item We write $x<_{\mathfrak{R}}y$ if $\exists r\in\mathfrak{R}:\ r\le y-x$. 
\item $B_{r}^{\mathfrak{R}}(c):=\left\{ x\in\RC{\rho}^{n}\mid\left|x-c\right|<_{\mathfrak{R}}r\right\} $
for each $r\in\mathfrak{R}$. 
\item $\Eball_{r}(c):=\{x\in\R^{n}\mid|x-c|<r\}$, for each $r\in\R_{>0}$,
denotes an ordinary Euclidean ball in $\R^{n}$. 
\end{enumerate}
\end{defn}
\noindent The relation $<_{\mathfrak{R}}$ has better topological
properties as compared to the usual strict order relation $a\le b$
and $a\ne b$ (that we will \emph{never} use) because for $\mathfrak{R}\in\left\{ \RC{\rho}_{\ge0}^{*},\R_{>0}\right\} $
the set of balls $\left\{ B_{r}^{\mathfrak{R}}(c)\mid r\in\mathfrak{R},\ c\in\RC{\rho}^{n}\right\} $
is a base for a topology on $\RC{\rho}^{n}$. The topology generated
in the case $\mathfrak{R}=\RC{\rho}_{\ge0}^{*}$ is called \emph{sharp
topology}, whereas the one with the set of radii $\mathfrak{R}=\R_{>0}$
is called \emph{Fermat topology}. We will call \emph{sharply open
set} any open set in the sharp topology, and \emph{large open set}
any open set in the Fermat topology; clearly, the latter is coarser
than the former. The existence of infinitesimal neighborhoods implies
that the sharp topology induces the discrete topology on $\R$. This
is a necessary result when one has to deal with continuous generalized
functions which have infinite derivatives. In fact, if $f'(x_{0})$
is infinite, we have $f(x)\approx f(x_{0})$ only for $x\approx x_{0}$
, see \cite{GiKu13,GK15}. With an innocuous abuse of language, we
write $x<y$ instead of $x<_{\RC{\rho}_{\ge0}^{*}}y$ and $x<_{\R}y$
instead of $x<_{\R_{>0}}y$. For example, $\RC{\rho}_{\ge0}^{*}=\RC{\rho}_{>0}$.
We will simply write $B_{r}(c)$ to denote an open ball in the sharp
topology and $B_{r}^{{\scriptscriptstyle \text{F}}}(c)$ for an open
ball in the Fermat topology. Also open intervals are defined using
the relation $<$, i.e.~$(a,b):=\{x\in\rc\mid a<x<b\}$.

The following result is useful to deal with positive and invertible
generalized numbers (cf.~\cite{GKOS}). 
\begin{lem}
\label{lem:mayer} Let $x\in\RC{\rho}$. Then the following are equivalent:

\begin{enumerate}
\item \label{enu:positiveInvertible}$x$ is invertible and $x\ge0$, i.e.~$x>0$. 
\item \label{enu:strictlyPositive}For each representative $(x_{\eps})\in\R_{\rho}$
of $x$ we have $\forall^{0}\eps:\ x_{\eps}>0$. 
\item \label{enu:greater-i_epsTom}For each representative $(x_{\eps})\in\R_{\rho}$
of $x$ we have $\exists m\in\N\,\forall^{0}\eps:\ x_{\eps}>\rho_{\eps}^{m}$ 
\end{enumerate}
\end{lem}
We will also need the following result.
\begin{lem}
\label{lem:approxOfBoundaryPointsWithInterior}Let $a$, $b\in\rcrho$
such that $a<b$, then the interior $\text{\emph{int}}\left([a,b]\right)$
in the sharp topology is dense in $[a,b]$.
\end{lem}

\subsubsection*{Internal and strongly internal sets}

A natural way to obtain sharply open, closed and bounded sets in $\RC{\rho}^{n}$
is by using a net $(A_{\eps})$ of subsets $A_{\eps}\subseteq\R^{n}$.
We have two ways of extending the membership relation $x_{\eps}\in A_{\eps}$
to generalized points $[x_{\eps}]\in\RC{\rho}$: 
\begin{defn}
\label{def:internalStronglyInternal}Let $(A_{\eps})$ be a net of
subsets of $\R^{n}$, then

\begin{enumerate}
\item $[A_{\eps}]:=\left\{ [x_{\eps}]\in\RC{\rho}^{n}\mid\forall^{0}\eps:\,x_{\eps}\in A_{\eps}\right\} $
is called the \emph{internal set} generated by the net $(A_{\eps})$.
See \cite{ObVe08} for the introduction and an in-depth study of this
notion in the case $\rho_{\eps}=\eps$. 
\item Let $(x_{\eps})$ be a net of points of $\R^{n}$, then we say that
$x_{\eps}\in_{\eps}A_{\eps}$, and we read it as $(x_{\eps})$ \emph{strongly
belongs to $(A_{\eps})$}, if $\forall^{0}\eps:\ x_{\eps}\in A_{\eps}$
and if $(x'_{\eps})\sim_{\rho}(x_{\eps})$, then also $x'_{\eps}\in A_{\eps}$
for $\eps$ small. Moreover, we set $\sint{A_{\eps}}:=\left\{ [x_{\eps}]\in\RC{\rho}^{n}\mid x_{\eps}\in_{\eps}A_{\eps}\right\} $,
and we call it the \emph{strongly internal set} generated by the net
$(A_{\eps})$. 
\item Finally, we say that the internal set $K=[A_{\eps}]$ is \emph{sharply
bounded} if there exists $r\in\RC{\rho}_{>0}$ such that $K\subseteq B_{r}(0)$.
Analogously, a net $(A_{\eps})$ is \emph{sharply bounded} the internal
set $[A_{\eps}]$ is sharply bounded. 
\end{enumerate}
\end{defn}
\noindent Therefore, $x\in[A_{\eps}]$ if there exists a representative
$[x_{\eps}]=x$ such that $x_{\eps}\in A_{\eps}$ for $\eps$ small,
whereas this membership is independent from the chosen representative
in the case of strongly internal sets. Note explicitly that an internal
set generated by a constant net $A_{\eps}=A\subseteq\R^{n}$ is simply
denoted by $[A]$.

The following theorem shows that internal and strongly internal sets
have dual topological properties: 
\begin{thm}
\noindent \label{thm:strongMembershipAndDistanceComplement}For $\eps\in I$,
let $A_{\eps}\subseteq\R^{n}$ and let $x_{\eps}\in\R^{n}$. Then
we have

\begin{enumerate}
\item \label{enu:internalSetsDistance}$[x_{\eps}]\in[A_{\eps}]$ if and
only if $\forall q\in\R_{>0}\,\forall^{0}\eps:\ d(x_{\eps},A_{\eps})\le\rho_{\eps}^{q}$.
Therefore $[x_{\eps}]\in[A_{\eps}]$ if and only if $[d(x_{\eps},A_{\eps})]=0\in\RC{\rho}$. 
\item \label{enu:stronglyIntSetsDistance}$[x_{\eps}]\in\sint{A_{\eps}}$
if and only if $\exists q\in\R_{>0}\,\forall^{0}\eps:\ d(x_{\eps},A_{\eps}^{c})>\rho_{\eps}^{q}$,
where $A_{\eps}^{c}:=\R^{n}\setminus A_{\eps}$. Therefore, if $(d(x_{\eps},A_{\eps}^{c}))\in\R_{\rho}$,
then $[x_{\eps}]\in\sint{A_{\eps}}$ if and only if $[d(x_{\eps},A_{\eps}^{c})]>0$. 
\item \label{enu:internalAreClosed}$[A_{\eps}]$ is sharply closed and
$\sint{A_{\eps}}$ is sharply open. 
\item \label{enu:internalGeneratedByClosed}$[A_{\eps}]=\left[\text{\emph{cl}}\left(A_{\eps}\right)\right]$,
where $\text{\emph{cl}}\left(S\right)$ is the closure of $S\subseteq\R^{n}$.
On the other hand $\sint{A_{\eps}}=\sint{\text{\emph{int}\ensuremath{\left(A_{\eps}\right)}}}$,
where $\emph{int}\left(S\right)$ is the interior of $S\subseteq\R^{n}$. 
\end{enumerate}
\end{thm}

\subsubsection*{Generalized smooth functions and their calculus}

Using the ring $\rti$, it is easy to consider a Gaussian with an
infinitesimal standard deviation. If we denote this probability density
by $f(x,\sigma)$, and if we set $\sigma=[\sigma_{\eps}]\in\RC{\rho}_{>0}$,
where $\sigma\approx0$, we obtain the net of smooth functions $(f(-,\sigma_{\eps}))_{\eps\in I}$.
This is the basic idea we are going to develop in the following 
\begin{defn}
\label{def:netDefMap}Let $X\subseteq\RC{\rho}^{n}$ and $Y\subseteq\RC{\rho}^{d}$
be arbitrary subsets of generalized points. Then we say that 
\[
f:X\longrightarrow Y\text{ is a \emph{generalized smooth function}}
\]
if there exists a net $f_{\eps}\in\cinfty(\Omega_{\eps},\R^{d})$
defining $f$ in the sense that $X\subseteq\langle\Omega_{\eps}\rangle$,
$f([x_{\eps}])=[f_{\eps}(x_{\eps})]\in Y$ and $(\partial^{\alpha}f_{\eps}(x_{\eps}))\in\R_{{\scriptscriptstyle \rho}}^{d}$
for all $x=[x_{\eps}]\in X$ and all $\alpha\in\N^{n}$. The space
of generalized smooth functions (GSF) from $X$ to $Y$ is denoted
by $\gsf(X,Y)$. 
\end{defn}
Let us note explicitly that this definition states minimal logical
conditions to obtain a set-theoretical map from $X$ into $Y$ and
defined by a net of smooth functions. In particular, the following
Thm.~\ref{thm:propGSF} states that the equality $f([x_{\eps}])=[f_{\eps}(x_{\eps})]$
is meaningful, i.e.~that we have independence from the representatives
for all derivatives $[x_{\eps}]\in X\mapsto[\partial^{\alpha}f_{\eps}(x_{\eps})]\in\RC{\rho}^{d}$,
$\alpha\in\N^{n}$. 
\begin{thm}
\label{thm:propGSF}Let $X\subseteq\RC{\rho}^{n}$ and $Y\subseteq\RC{\rho}^{d}$
be arbitrary subsets of generalized points. Let $f_{\eps}\in\cinfty(\Omega_{\eps},\R^{d})$
be a net of smooth functions that defines a generalized smooth map
of the type $X\longrightarrow Y$, then

\begin{enumerate}
\item $\forall\alpha\in\N^{n}\,\forall(x_{\eps}),(x'_{\eps})\in\R_{\rho}^{n}:\ [x_{\eps}]=[x'_{\eps}]\in X\ \Rightarrow\ (\partial^{\alpha}u_{\eps}(x_{\eps}))\sim_{\rho}(\partial^{\alpha}u_{\eps}(x'_{\eps}))$. 
\item \label{enu:modOnEpsDepBall}$\forall[x_{\eps}]\in X\,\forall\alpha\in\N^{n}\,\exists q\in\R_{>0}\,\forall^{0}\eps:\ \sup_{y\in\Eball_{\eps^{q}}(x_{\eps})}\left|\partial^{\alpha}u_{\eps}(y)\right|\le\eps^{-q}$. 
\item \label{enu:locLipSharp}For all $\alpha\in\N^{n}$, the GSF $g:[x_{\eps}]\in X\mapsto[\partial^{\alpha}f_{\eps}(x_{\eps})]\in\Rtil^{d}$
is locally Lipschitz in the sharp topology, i.e.~each $x\in X$ possesses
a sharp neighborhood $U$ such that $|g(x)-g(y)|\le L|x-y|$ for all
$x$, $y\in U$ and some $L\in\RC{\rho}$. 
\item \label{enu:GSF-cont}Each $f\in\gsf(X,Y)$ is continuous with respect
to the sharp topologies induced on $X$, $Y$. 
\item \label{enu:suffCondFermatCont}Assume that the GSF $f$ is locally
Lipschitz in the Fermat topology and that its Lipschitz constants
are always finite: $L\in\R$. Then $f$ is continuous in the Fermat
topology. 
\item \label{enu:globallyDefNet}$f:X\longrightarrow Y$ is a GSF if and
only if there exists a net $v_{\eps}\in\cinfty(\R^{n},\R^{d})$ defining
a generalized smooth map of type $X\longrightarrow Y$ such that $f=[v_{\eps}(-)]|_{X}$. 
\item \label{enu:category}Subsets $S\subseteq\RC{\rho}^{s}$ with the trace
of the sharp topology, and generalized smooth maps as arrows form
a subcategory of the category of topological spaces. We will call
this category $\gsf$, the \emph{category of GSF}. 
\end{enumerate}
\end{thm}
The differential calculus for GSF can be introduced showing existence
and uniqueness of another GSF serving as incremental ratio.
\begin{thm}[Fermat-Reyes theorem for GSF]
\noindent \label{thm:FR-forGSF} Let $U\subseteq\RC{\rho}^{n}$ be
a sharply open set, let $v=[v_{\eps}]\in\RC{\rho}^{n}$, and let $f\in\gsf(U,\RC{\rho})$
be a generalized smooth map generated by the net of smooth functions
$f_{\eps}\in\cinfty(\Omega_{\eps},\R)$. Then

\begin{enumerate}
\item \label{enu:existenceRatio}There exists a sharp neighborhood $T$
of $U\times\{0\}$ and a generalized smooth map $r\in\gsf(T,\RC{\rho})$,
called the \emph{generalized incremental ratio} of $f$ \emph{along}
$v$, such that
\[
\forall(x,h)\in T:\ f(x+hv)=f(x)+h\cdot r(x,h).
\]
\item \label{enu:uniquenessRatio}Any two generalized incremental ratios
coincide on a sharp neighborhood of $U\times\{0\}$.
\item \label{enu:defDer}We have $r(x,0)=\left[\frac{\partial f_{\eps}}{\partial v_{\eps}}(x_{\eps})\right]$
for every $x\in U$ and we can thus define $Df(x)\cdot v:=\frac{\partial f}{\partial v}(x):=r(x,0)$,
so that $\frac{\partial f}{\partial v}\in\gsf(U,\RC{\rho})$. 
\end{enumerate}
\noindent If $U$ is a large open set, then an analogous statement
holds replacing sharp neighborhoods by large neighborhoods.
\end{thm}
Note that this result permits to consider the partial derivative of
$f$ with respect to an arbitrary generalized vector $v\in\RC{\rho}^{n}$
which can be, e.g., infinitesimal or infinite. Using this result,
we can also define subsequent differentials $D^{j}f(x)$ as $j-$multilinear
maps, and we set $D^{j}f(x)\cdot h^{j}:=D^{j}f(x)(h,\ptind^{j},h)$.
The set of all the $j-$multilinear maps $\left(\rti^{n}\right)^{j}\ra\rti^{d}$
over the ring $\rti$ will be denoted by $L^{j}(\rti^{n},\rti^{d})$.
For $A=[A_{\eps}(-)]\in L^{j}(\rti^{n},\rti^{d})$, we set $\no{A}:=[\no{A_{\eps}}]$,
the generalized number defined by the operator norms of the multilinear
maps $A_{\eps}\in L^{j}(\R^{n},\R^{d})$.

The following result follows from the analogous properties for the
nets of smooth functions defining $f$ and $g$.
\begin{thm}
\label{thm:rulesDer} Let $U\subseteq\rcrho^{n}$ be an open subset
in the sharp topology, let $v\in\rcrho^{n}$ and $f$, $g:U\longrightarrow\rcrho$
be generalized smooth maps. Then

\begin{enumerate}
\item $\frac{\partial(f+g)}{\partial v}=\frac{\partial f}{\partial v}+\frac{\partial g}{\partial v}$
\item $\frac{\partial(r\cdot f)}{\partial v}=r\cdot\frac{\partial f}{\partial v}\quad\forall r\in\rcrho$
\item $\frac{\partial(f\cdot g)}{\partial v}=\frac{\partial f}{\partial v}\cdot g+f\cdot\frac{\partial g}{\partial v}$
\item For each $x\in U$, the map $\diff{f}(x).v:=\frac{\partial f}{\partial v}(x)\in\rcrho$
is $\rcrho$-linear in $v\in\rcrho^{n}$
\item Let $U\subseteq\rcrho^{n}$ and $V\subseteq\rcrho^{d}$ be open subsets
in the sharp topology and $g\in{}^{\rho}\Gcinf(V,U)$, $f\in{}^{\rho}\Gcinf(U,\rcrho)$
be generalized smooth maps. Then for all $x\in V$ and all $v\in\rcrho^{d}$,
we have $\frac{\partial\left(f\circ g\right)}{\partial v}(x)=\diff{f}\left(g(x)\right).\frac{\partial g}{\partial v}(x)$.
\end{enumerate}
\end{thm}
We also have a generalization of the Taylor formula:
\begin{thm}
\label{thm:Taylor}Let $f\in{}^{\rho}\Gcinf(U,\rcrho)$ be a generalized
smooth function defined in the sharply open set $U\subseteq\rcrho^{n}$.
Let $a$, $x\in\rcrho^{n}$ such that the line segment $[a,x]\subseteq U$.
Then, for all $n\in\N$ we have
\begin{equation}
\exists\xi\in[a,x]:\ f(x)=\sum_{j=0}^{n}\frac{D^{j}f(a)}{j!}\cdot(x-a)^{j}+\frac{D^{n+1}f(\xi)}{(n+1)!}\cdot(x-a)^{n+1}.\label{eq:lagrange}
\end{equation}
If we further assume that all the $n$ components $(x-a)_{k}\in\rcrho$
of $x-a\in\rcrho^{n}$ are invertible, then there exists $\rho\in\rcrho_{>0}$,
$\rho\le|x-a|$, such that 
\begin{equation}
\forall k\in B_{\rho}(0)\,\exists\xi\in[a-k,a+k]:\ f(a+k)=\sum_{j=0}^{n}\frac{D^{j}f(a)}{j!}\cdot k^{j}+\frac{D^{n+1}f(\xi)}{(n+1)!}\cdot k^{n+1}\label{eq:LagrangeInfRest}
\end{equation}
\begin{equation}
\frac{D^{n+1}f(\xi)}{(n+1)!}\cdot k^{n+1}\approx0.\label{eq:integralInfRest}
\end{equation}
\end{thm}
Formula \eqref{eq:lagrange} corresponds to a direct generalization of Taylor formulas for ordinary smooth functions with Lagrange remainder.
On the other hand, in \eqref{eq:LagrangeInfRest} and \eqref{eq:integralInfRest},
the possibility that the differential $D^{n+1}f$ may be infinite
at some point is considered, and the Taylor formulas are stated so
as to have infinitesimal remainder.

The following local inverse function theorem will be used in the proof
of Jacobi's theorem (see \cite{GiKu16} for a proof).
\begin{thm}
\label{thm:localIFTSharp}Let $X\sse\rti^{n}$, let $f\in\gsf(X,\rti^{n})$
and suppose that for some $x_{0}$ in the sharp interior of $X$,
$Df(x_{0})$ is invertible in $L(\rti^{n},\rti^{n})$. Then there
exists a sharp neighborhood $U\sse X$ of $x_{0}$ and a sharp neighborhood
$V$ of $f(x_{0})$ such that $f:U\to V$ is invertible and $f^{-1}\in\gsf(V,U)$. 
\end{thm}
We can define right and left derivatives as e.g.~$f'(a):=f'_{+}(a):=\lim_{\substack{t\to a\\
a<t
}
}f'(t)$, which always exist if $f\in\gsf([a,b],\rc^{d})$. One dimensional
integral calculus of GSF is based on the following
\begin{thm}
\label{thm:existenceUniquenessPrimitives}Let $f\in{}^{\rho}\Gcinf([a,b],\rcrho)$
be a generalized smooth function defined in the interval $[a,b]\subseteq\rc$,
where $a<b$. Let $c\in[a,b]$. Then, there exists one and only one
generalized smooth function $F\in{}^{\rho}\Gcinf([a,b],\rcrho)$ such
that $F(c)=0$ and $F'(x)=f(x)$ for all $x\in[a,b]$. Moreover, if
$f$ is defined by the net $f_{\eps}\in\Coo(\R,\R)$ and $c=[c_{\eps}]$,
then $F(x)=\left[\int_{c_{\eps}}^{x_{\eps}}f_{\eps}(s)\diff{s}\right]$
for all $x=[x_{\eps}]\in[a,b]$.
\end{thm}
\noindent We can thus define
\begin{defn}
\label{def:integral}Under the assumptions of Theorem \ref{thm:existenceUniquenessPrimitives},
we denote by $\int_{c}^{(-)}f:=\int_{c}^{(-)}f(s)\diff{s}\in{}^{\rho}\Gcinf([a,b],\rcrho)$
the unique generalized smooth function such that:

\begin{enumerate}
\item $\int_{c}^{c}f=0$
\item $\left(\int_{u}^{(-)}f\right)'(x)=\frac{\diff{}}{\diff{x}}\int_{u}^{x}f(s)\diff{s}=f(x)$
for all $x\in[a,b]$.
\end{enumerate}
\end{defn}
\noindent All the classical rules of integral calculus hold in this
setting:
\begin{thm}
\label{thm:intRules}Let $f\in{}^{\rho}\Gcinf(U,\rcrho)$ and $g\in{}^{\rho}\Gcinf(V,\rcrho)$
be generalized smooth functions defined on sharply open domains in
$\rcrho$. Let $a$, $b\in\rcrho$ with $a<b$ and $c$, $d\in[a,b]\subseteq U\cap V$,
then

\begin{enumerate}
\item \label{enu:additivityFunction}$\int_{c}^{d}\left(f+g\right)=\int_{c}^{d}f+\int_{c}^{d}g$ 
\item \label{enu:homog}$\int_{c}^{d}\lambda f=\lambda\int_{c}^{d}f\quad\forall\lambda\in\rcrho$ 
\item \label{enu:additivityDomain}$\int_{c}^{d}f=\int_{c}^{e}f+\int_{e}^{d}f$
for all $e\in[a,b]$ 
\item \label{enu:chageOfExtremes}$\int_{c}^{d}f=-\int_{d}^{c}f$ 
\item \label{enu:foundamental}$\int_{c}^{d}f'=f(d)-f(c)$ 
\item \label{enu:intByParts}$\int_{c}^{d}f'\cdot g=\left[f\cdot g\right]_{c}^{d}-\int_{c}^{d}f\cdot g'$ 
\end{enumerate}
\end{thm}
~
\begin{thm}
\label{thm:changeOfVariablesInt}Let $f\in{}^{\rho}\Gcinf(U,\rcrho)$
and $\phi\in{}^{\rho}\Gcinf(V,U)$ be generalized smooth functions
defined on sharply open domains in $\rcrho$. Let $a$, $b\in\rcrho$,
with $a<b$, such that $[a,b]\subseteq V$, $\phi(a)<\phi(b)$, $[\phi(a),\phi(b)]\subseteq U$.
Finally, assume that $\phi([a,b])\subseteq[\phi(a),\phi(b)]$. Then
\[
\int_{\phi(a)}^{\phi(b)}f(t)\diff{t}=\int_{a}^{b}f\left[\phi(s)\right]\cdot\phi'(s)\diff{s}.
\]
\end{thm}

\subsubsection*{Embedding of Schwartz distributions and Colombeau functions}

We finally recall two results that give a certain flexibility in constructing
embeddings of Schwartz distributions. Note that both the infinitesimal
$\rho$ and the embedding of Schwartz distributions have to be chosen
depending on the problem we aim to solve. A trivial example in this
direction is the ODE $y'=y/\diff{\eps}$, which cannot be solved for
$\rho=(\eps)$, but it has a solution for $\rho=(e^{-1/\eps})$. As
another simple example, if we need the property $H(0)=1/2$, where
$H$ is the Heaviside function, then we have to choose the embedding
of distributions accordingly. See also \cite{GiLu15,Lu-Gi16b} for
further details.\\
 If $\phi\in\mathcal{D}(\R^{n})$, $r\in\R_{>0}$ and $x\in\R^{n}$,
we use the notations $r\odot\phi$ for the function $x\in\R^{n}\mapsto\frac{1}{r^{n}}\cdot\phi\left(\frac{x}{r}\right)\in\R$
and $x\oplus\phi$ for the function $y\in\R^{n}\mapsto\phi(y-x)\in\R$.
These notations permit to highlight that $\odot$ is a free action
of the multiplicative group $(\R_{>0},\cdot,1)$ on $\mathcal{D}(\R^{n})$
and $\oplus$ is a free action of the additive group $(\R_{>0},+,0)$
on $\mathcal{D}(\R^{n})$. We also have the distributive property
$r\odot(x\oplus\phi)=rx\oplus r\odot\phi$. 
\begin{lem}
\label{lem:strictDeltaNet}Let $b\in\rti$
be a net such that $\lim_{\eps\to0^{+}}b_{\eps}=+\infty$. Let $d\in(0,1)$.
There exists a net $\left(\psi_{\eps}\right)_{\eps\in I}$ of $\mathcal{D}(\R^{n})$
with the properties:

\begin{enumerate}
\item \label{enu:suppStrictDeltaNet}$supp(\psi_{\eps})\subseteq B_{1}(0)$
for all $\eps\in I$. 
\item \label{enu:intOneStrictDeltaNet}$\int\psi_{\eps}=1$ for all $\eps\in I$. 
\item \label{enu:moderateStrictDeltaNet}$\forall\alpha\in\N^{n}\,\exists p\in\N:\ \sup_{x\in\R^{n}}\left|\partial^{\alpha}\psi_{\eps}(x)\right|=O(b_{\eps}^{p})$
as $\eps\to0^{+}$. 
\item \label{enu:momentsStrictDeltaNet}$\forall j\in\N\,\forall^{0}\eps:\ 1\le|\alpha|\le j\Rightarrow\int x^{\alpha}\cdot\psi_{\eps}(x)\diff{x}=0$. 
\item \label{enu:smallNegPartStrictDeltaNet}$\forall\eta\in\R_{>0}\,\forall^{0}\eps:\ \int\left|\psi_{\eps}\right|\le1+\eta$. 
\item \label{enu:int1Dim}If $n=1$, then the net $(\psi_{\eps})_{\eps\in I}$
can be chosen so that $\int_{-\infty}^{0}\psi_{\eps}=d$. 
\end{enumerate}
\noindent In particular $\psi_{\eps}^{b}:=b_{\eps}^{-1}\odot\psi_{\eps}$
satisfies \ref{enu:intOneStrictDeltaNet} - \ref{enu:smallNegPartStrictDeltaNet}. 
\end{lem}
It is worth noting that the condition \ref{enu:momentsStrictDeltaNet}
of null moments is well known in the study of convergence of numerical
solutions of singular differential equations, see e.g.~\cite{To-En04,En-To-Ts05,Ho-Ni-St16}
and references therein.

\noindent Concerning embeddings of Schwartz distributions, we have
the following result, where $\csp{\Omega}:=\{[x_{\eps}]\in[\Omega]\mid\exists K\Subset\Omega\,\forall^{0}\eps:\ x_{\eps}\in K\}$
is called the set of \emph{compactly supported points in }$\Omega\subseteq\R^{n}$. 
\begin{thm}
\label{thm:embeddingD'}Under the assumptions of Lemma \ref{lem:strictDeltaNet},
let $\Omega\subseteq\R^{n}$ be an open set and let $(\psi_{\eps}^{b})$
be the net defined in \ref{lem:strictDeltaNet}. Then the mapping
\[
\iota_{\Omega}^{b}:T\in\mathcal{E}'(\Omega)\mapsto\left[\left(T\ast\psi_{\eps}^{b}\right)(-)\right]\in\gsf(\csp{\Omega},\rti)
\]
uniquely extends to a sheaf morphism of real vector spaces 
\[
\iota^{b}:\mathcal{D}'\ra\gsf(\csp{(-)},\rti),
\]
and satisfies the following properties:

\begin{enumerate}
\item If $b\ge\diff{\rho}^{-a}$ for some $a\in\R_{>0}$, then $\iota^{b}|_{\Coo(-)}:\Coo(-)\ra\gsf(\csp{(-)},\RC{\rho})$
is a sheaf morphism of algebras; 
\item If $T\in\mathcal{E}'(\Omega)$ then $\text{\text{\emph{supp}}}(T)=\text{\emph{\text{supp}}}(\iota_{\Omega}^{b}(T))$; 
\item \label{enu:eps-D'}$\lim_{\eps\to0^{+}}\int_{\Omega}\iota_{\Omega}^{b}(T)_{\eps}\cdot\phi=\langle T,\phi\rangle$
for all $\phi\in\mathcal{D}(\Omega)$ and all $T\in\mathcal{D}'(\Omega)$; 
\item $\iota^{b}$ commutes with partial derivatives, i.e.~$\partial^{\alpha}\left(\iota_{\Omega}^{b}(T)\right)=\iota_{\Omega}^{b}\left(\partial^{\alpha}T\right)$
for each $T\in\mathcal{D}'(\Omega)$ and $\alpha\in\N$. 
\end{enumerate}
\end{thm}
Concerning the embedding of Colombeau generalized functions, we recall
that the special Colombeau algebra on $\Om$ is defined as the quotient
$\gs(\Om):=\esm(\Om)/\ns(\Om)$ of \emph{moderate nets} over \emph{negligible
nets}, where the former is 
\begin{multline*}
\esm(\Om):=\{(u_{\eps})\in\cinfty(\Omega)^{I}\mid\forall K\comp\Om\,\forall\alpha\in\N^{n}\,\exists N\in\N:\sup_{x\in K}|\partial^{\alpha}u_{\eps}(x)|=O(\eps^{-N})\}
\end{multline*}
and the latter is 
\begin{multline*}
\ns(\Om):=\{(u_{\eps})\in\cinfty(\Omega)^{I}\mid\forall K\comp\Om\,\forall\alpha\in\N^{n}\,\forall m\in\N:\sup_{x\in K}|\partial^{\alpha}u_{\eps}(x)|=O(\eps^{m})\}.
\end{multline*}
Using $\rho=(\eps)$, we have the following compatibility result: 
\begin{thm}
\label{thm:inclusionCGF}A Colombeau generalized function $u=(u_{\eps})+\ns(\Om)^{d}\in\gs(\Omega)^{d}$
defines a generalized smooth map $u:[x_{\eps}]\in\csp{\Omega}\longrightarrow[u_{\eps}(x_{\eps})]\in\Rtil^{d}$
which is locally Lipschitz on the same neighborhood of the Fermat
topology for all derivatives. This assignment provides a bijection
of $\gs(\Omega)^{d}$ onto $\gsf(\csp{\Omega},\rti^{d})$ for every
open set $\Omega\subseteq\R^{n}$.
\end{thm}

\subsection{Extreme value theorem and functionally compact sets}

For GSF, suitable generalizations of many classical theorems of differential
and integral calculus hold: intermediate value theorem, mean value
theorems, a sheaf property for the Fermat topology, local and global
inverse function theorems, Banach fixed point theorem and a corresponding
Picard-Lindelöf theorem, see \cite{GKV,Gi-Ku-St15,Lu-Gi16,GiKu16}.

Even though the intervals $[a,b]\subseteq\Rtil$, $a$, $b\in\R$,
are neither compact in the sharp nor in the Fermat topology (see \cite[Thm. 25]{GKV}),
analogously to the case of smooth functions, a GSF satisfies an extreme
value theorem on such sets. In fact, we have: 
\begin{thm}
\label{thm:extremeValues}Let $f\in\Gcinf(X,\Rtil)$ be a generalized
smooth function defined on the subset $X$ of $\Rtil^{n}$. Let $\emptyset\ne K=[K_{\eps}]\subseteq X$
be an internal set generated by a sharply bounded net $(K_{\eps})$
of compact sets $K_{\eps}\comp\R^{n}$ , then 
\begin{equation}
\exists m,M\in K\,\forall x\in K:\ f(m)\le f(x)\le f(M).\label{eq:epsExtreme}
\end{equation}
\end{thm}
We shall use the assumptions on $K$ and $(K_{\eps})$ given in this
theorem to introduce a notion of ``compact subset'' which behaves
better than the usual classical notion of compactness in the sharp
topology. 
\begin{defn}
\label{def:functCmpt} A subset $K$ of $\Rtil^{n}$ is called \emph{functionally
compact}, denoted by $K\fcmp\Rtil^{n}$, if there exists a net $(K_{\eps})$
such that

\begin{enumerate}
\item \label{enu:defFunctCmpt-internal}$K=[K_{\eps}]\subseteq\Rtil^{n}$ 
\item \label{enu:defFunctCmpt-sharpBound}$(K_{\eps})$ is sharply bounded 
\item \label{enu:defFunctCmpt-cmpt}$\forall\eps\in I:\ K_{\eps}\Subset\R^{n}$ 
\end{enumerate}
If, in addition, $K\subseteq U\subseteq\Rtil^{n}$ then we write $K\fcmp U$.
Finally, we write $[K_{\eps}]\fcmp U$ if \ref{enu:defFunctCmpt-sharpBound},
\ref{enu:defFunctCmpt-cmpt} and $[K_{\eps}]\subseteq U$ hold.
\end{defn}
\noindent We motivate the name \emph{functionally compact subset}
by noting that on this type of subsets, GSF have properties very similar
to those that ordinary smooth functions have on standard compact sets. 
\begin{rem}
\noindent \label{rem:defFunctCmpt}\ 

\begin{enumerate}
\item \label{enu:rem-defFunctCmpt-closed}By \cite[Prop. 2.3]{ObVe08},
any internal set $K=[K_{\eps}]$ is closed in the sharp topology.
In particular, the open interval $(0,1)\subseteq\Rtil$ is not functionally
compact since it is not closed. 
\item \label{enu:rem-defFunctCmpt-ordinaryCmpt}If $H\Subset\R^{n}$ is
a non-empty ordinary compact set, then $[H]$ is functionally compact.
In particular, $[0,1]=\left[[0,1]_{\R}\right]$ is functionally compact.
\item \label{enu:rem-defFunctCmpt-empty}The empty set $\emptyset=\widetilde{\emptyset}\fcmp\Rtil$. 
\item \label{enu:rem-defFunctCmpt-equivDef}$\Rtil^{n}$ is not functionally
compact since it is not sharply bounded.
\item \label{enu:rem-defFunctCmpt-cmptlySuppPoints}The set of compactly
supported points $\Rtil_{c}$ is not functionally compact because
the GSF $f(x)=x$ does not satisfy the conclusion \eqref{eq:epsExtreme}
of Prop.~\ref{thm:extremeValues}.
\end{enumerate}
\end{rem}
\noindent In the present paper, we need the following properties of
functionally compact sets. 
\begin{thm}
\label{thm:image}Let $K\subseteq X\subseteq\Rtil^{n}$, $f\in\Gcinf(X,\Rtil^{d})$.
Then $K\fcmp\Rtil^{n}$ implies $f(K)\fcmp\Rtil^{d}$.
\end{thm}
\noindent As a corollary of this theorem and Rem.\ \eqref{rem:defFunctCmpt}.\ref{enu:rem-defFunctCmpt-ordinaryCmpt}
we get 
\begin{cor}
\label{cor:intervalsFunctCmpt}If $a$, $b\in\Rtil$ and $a\le b$,
then $[a,b]\fcmp\Rtil$.
\end{cor}
\noindent Let us note that $a$, $b\in\Rtil$ can also be infinite,
e.g.~$a=[-\eps^{-N}]$, $b=[\eps^{-M}]$ or $a=[\eps^{-N}]$, $b=[\eps^{-M}]$
with $M>N$. Finally, in the following result we consider the product
of functionally compact sets: 
\begin{thm}
\noindent \label{thm:product}Let $K\fcmp\Rtil^{n}$ and $H\fcmp\Rtil^{d}$,
then $K\times H\fcmp\Rtil^{n+d}$. In particular, if $a_{i}\le b_{i}$
for $i=1,\ldots,n$, then $\prod_{i=1}^{n}[a_{i},b_{i}]\fcmp\Rtil^{n}$.
\end{thm}
A theory of compactly supported GSF has been developed in \cite{GiKu16},
and it closely resembles the classical theory of LF-spaces of compactly
supported smooth functions. It establishes that for suitable functionally
compact subsets, the corresponding space of compactly supported GSF
contains extensions of all Colombeau generalized functions, and hence
also of all Schwartz distributions.

\section{\label{sec:Preliminary-results}Preliminary results for calculus
of variations with GSF}

In this section, we study extremal values of generalized functions
at sharply interior points of intervals $[a,b]\subseteq\rc$. As in
the classical calculus of variations, this will provide the basis
for proving necessary and sufficient conditions for general variational
problems. Since the new ring of scalars $\rc$ has zero divisors and
is not totally ordered, the following extension requires a more refined
analysis than in the classical case.

The following lemma shows that we can interchange integration and
differentiation while working with generalized functions. 
\begin{lem}
\label{lem:int}Let $a$, $b$, $c$, $d\in\rc$, with $a<b$ and
$c<d$. Let $f\in\gsf(X,Y)$ and assume that $[a,b]\times[c,d]\subseteq X\subseteq\rc^{2}$
and $Y\subseteq\rc^{d}$. Then for all $s\in[c,d]$,we have 
\begin{align}
\der{s}\int_{a}^{b}f(\tau,s)\,\dd\tau=\int_{a}^{b}\pdd{s}f(\tau,s)\,\dd\tau.\label{eq:derivUnderIntSign}
\end{align}
\end{lem}
\begin{proof}
We first note that $f(\cdot,s)\in\gsf([a,b],Y)$ by the closure of
GSF with respect to composition. Therefore, $\frac{\partial}{\partial s}f(\cdot,s)\in\gsf([a,b],\rc^{d})$,
and the right hand side of \eqref{eq:derivUnderIntSign} is well defined
as an integral of a GSF. In order to show that also the left hand side
of \eqref{eq:derivUnderIntSign} is well-defined, we need to prove
that also $\sigma\in[c,d]\mapsto\int_{a}^{b}f(\tau,\sigma)\diff{\tau}\in\rc^{d}$
is a GSF. Let $f$ be defined by the net $f_{\ve}\in\mathcal{C}^{\infty}\left(\Omega_{\ve},\R^{d}\right)$,
with $X\subseteq\left\langle \Omega_{\ve}\right\rangle $. Let $[\sigma_{\eps}]\in[c,d]$,
then $[a,b]\times\{[\sigma_{\eps}]\}\fcmp\rc^{2}$ and the extreme
value theorem \ref{thm:extremeValues} applied to $\frac{\partial^{n}f}{\partial\sigma^{n}}$
yields the existence of $N\in\R_{>0}$ such that 
\begin{align*}
\left|\frac{\text{d}^{n}}{\diff{\sigma}^{n}}\int_{a_\varepsilon}^{b_\varepsilon}f_{\ve}(\tau,\sigma_{\eps})\,\dd\tau\right|\le\int_{a_\varepsilon}^{b_\varepsilon}\left|\frac{\partial^{n}}{\partial\sigma^{n}}f_{\ve}(\tau,\sigma_{\eps})\right|\,\dd\tau\le\rho_{\eps}^{-N}\cdot(b_\varepsilon-a_\varepsilon).
\end{align*}
This proves that also the left hand side of \eqref{eq:derivUnderIntSign}
is well-defined as a derivative of a GSF. From the classical derivation
under the integral sign, the Fermat-Reyes theorem \ref{thm:FR-forGSF},
and Thm.~ \ref{thm:existenceUniquenessPrimitives} about definite
integrals of GSF, we obtain
\begin{align*}
\der{s}\int_{a}^{b}f(\tau,s)\,\dd\tau & =\frac{\text{d}}{\diff{s}}\int_{a}^{b}\left[f_{\eps}(\tau,s)\right]\,\diff{\tau}\\
 & =\frac{\text{d}}{\diff{s}}\left[\int_{a_{\eps}}^{b_{\eps}}f_{\eps}(\tau,s)\,\diff{\tau}\right]\\
 & =\left[\frac{\text{d}}{\diff{s}}\int_{a_{\eps}}^{b_{\eps}}f_{\eps}(\tau,s)\,\diff{\tau}\right]\\
 & =\int_{a}^{b}\left[\pdd{s}f_{\ve}(\tau,s)\right]\,\dd\tau\\
 & =\int_{a}^{b}\pdd{s}f(\tau,s)\,\dd\tau.
\end{align*}
\end{proof}
The next result will frequently be used in the following
\begin{lem}
\label{lem:permSignLimit}Let $(D,\ge)$ be a directed set and let
$f:D\ra\rc$ be a set-theoretical map such that $f(d)\ge0$ for all
$d\in D$, and $\exists\lim_{d\in D}f(d)\in\rc$ in the sharp topology.
Then $\lim_{d\in D}f(d)\ge0$.
\end{lem}
\begin{proof}
Note that the internal set $[0,+\infty)=\left[[0,+\infty)_{\R}\right]$ is sharply
closed by Thm.~\ref{thm:strongMembershipAndDistanceComplement}.\ref{enu:internalAreClosed}.
\end{proof}
\clearpage
\begin{rem}
\label{rem:limits}~

\begin{enumerate}
\item \label{enu:def_geq_zero}If $x\in\rc$, then $x\ge0$ if and only
if $\exists A\in\R_{>0}\,\forall a\in\R_{>A}\colon x\geq-\dd\rho^{a}.$
Indeed, it suffices to let $a\to+\infty$ in $f(a)=x+\diff{\rho}^{a}$.
\item \label{enu:x_is_zero}Assume that $x$, $y\in\rc^{n}$ and
\[
\exists s_{0}\in\rc_{>0}\,\forall s\in\rc_{>0}:\ s\leq s_{0}\ \Rightarrow\ |x|\leq s|y|.
\]
Then taking $s\to0$ in $f(s)=s|y|-|x|$ we get $x=0$.
\end{enumerate}
\end{rem}
\begin{defn}
We call $x=(x_{1},\ldots,x_{d})\in\rc^{d}$ \emph{componentwise invertible}
if and only if for all $k\in\{1,\ldots,d\}$ we have that $x_{k}\in\rc$
is invertible.
\end{defn}
\begin{lem}
\label{lem:componentWisePos}Let $f\in\gsf(U,Y)$ where $Y\subseteq\rc$
and $U\subseteq\rc^{d}$ is a sharply open subset. Then $f\geq0$
if and only if $f(x)\geq0$ for all componentwise invertible $x\in U$. 
\end{lem}
\begin{proof}
By Lem.~\ref{lem:mayer}, it follows that for $V\subseteq\rc$, the
set of invertible points in $V,$ i.e.~$V\cap\rc^{*}\subseteq V$
is dense in $V$ (with respect to the sharp topology). This implies
that $U\cap(\rc^{d})^{*}\subseteq U$ is dense. By Thm.~\ref{thm:propGSF}.\ref{enu:GSF-cont}, $f$
is sharply continuous, so Lem.~ \ref{lem:permSignLimit} yields that $f(x)\geq 0$. The other direction is obvious.
\end{proof}
Analogously to the classical case, we say that $x_{0}\in X$ is a
local minimum of $f\in\gsf(X)$ if there exists a sharply open neighbourhood
(in the trace topology) $Y\subseteq X$ of $x_{0}$ such that $f(x_{0})\leq f(y)$
for all $y\in Y$. A local maximum is defined accordingly. We will
write $f(x_{0})=\text{min}!$, which is a short hand notation to denote
that $x_{0}$ is a (local) minimum of $f$.
\begin{lem}
\label{lem:local_min_diff}Let $X\subseteq\rc$ and let $f\in\gsf(X,\rc)$.
If $x_{0}\in X$ is a sharply interior local minimum of $f$ then
$f'(x_{0})=0$. 
\end{lem}
\begin{proof}
Without loss of generality, we can assume $x_{0}=0$, because of the
closure of GSF with respect to composition. Let $r\in\rc_{>0}$ be
such that $B_{2r}(0)=:U\subseteq X$ and $f(0)=\text{min!}$ over
$U$. Take any $x\in\rc$ such that $0<|x|<r$, so that $[-|x|,|x|]\subseteq U$.
Thus, if $x>0$, by Taylor's theorem \ref{thm:Taylor} there exists
$\xi\in[0,x]$ such that 
\[
f(x)=f(0)+f'(0)\cdot x+\frac{f''(\xi)}{2}\cdot x^{2}.
\]
Set $K:=[B_{r_{\eps}}(0)]\fcmp B_{2r}(0)\subseteq U$ and $M:=\max_{x\in K}\left|f''(x)\right|\in\rc_{\ge0}$.
Due to the fact that $f(0)$ is minimal, we have
\begin{equation}
f'(0)\cdot x+\frac{f''(\xi)}{2}\cdot x^{2}=f(x)-f(0)\geq0.\label{eq:Df}
\end{equation}
Thus $-f'(0)\cdot x\le\frac{M}{2}x^{2}$ and $-f'(0)\le\frac{M}{2}|x|$
since $x>0$. Analogously, if we take $x<0$, we get $f'(0)\le-\frac{M}{2}x=\frac{M}{2}|x|$.
Therefore, $\left|f'(0)\right|\le\frac{M}{2}|x|$ and the conclusion
follows by Rem.~\ref{rem:limits}.\ref{enu:x_is_zero}.
\end{proof}
As a corollary of Lem.~\ref{lem:approxOfBoundaryPointsWithInterior}
and Thm.~\ref{thm:propGSF}.\ref{enu:GSF-cont}, we have
\begin{lem}
\label{lem:f_zero_on_sharp_int}Let $a$, $b\in\rc$ with $a<b$ and
let $f\in\gsf([a,b],\rc^{d})$ such that $f(x)=0$ for all sharply
interior points $x\in[a,b]$. Then $f=0$ on $[a,b]$. 
\end{lem}
\noindent Now, we are able to prove the ``second - derivative - test''
for GSF. 
\begin{lem}
\label{lem:second_deriv_test}Let $a$, $b\in\rc$ with $a<b$ and
let $f\in\gsf([a,b],\rc)$ such that $f(x_{0})=\text{min!}$ for some
sharply interior $x_{0}\in[a,b]$. Then $f''(x_{0})\geq0$. Vice versa,
if $f'(x_{0})=0$ and $f''(x_{0})>0$, then $f(x_{0})=\text{min!}$.
\end{lem}
\begin{proof}
As above, we can assume $x_{0}=0$. Let $r\in\rc_{>0}$ be such that
$B_{2r}(0)=:U\subseteq X$ and $f(0)=\text{min!}$ over $U$. Take
any $x\in\rc$ such that $0<x<r$, so that $[0,x]\subseteq U$, and
set $K:=[B_{r_{\eps}}(0)]\fcmp B_{2r}(0)\subseteq U$ and $M:=\max_{x\in K}\left|f'''(x)\right|\in\rc_{\ge0}$.
By Taylor's theorem \ref{thm:Taylor}, we obtain that for some $\xi\in[0,x]$
\begin{align*}
f(x)=f(0)+f'(0)x+\frac{1}{2}f''(0)x^{2}+\frac{1}{6}f'''(\xi)x^{3}.
\end{align*}
By assumption, we have that for all $a\in\R_{>0}$ 
\begin{align*}
0\leq f(x)-f(0)+\dd\rho^{a}.
\end{align*}
By Lemma \ref{lem:local_min_diff}, we know that $f'(0)=0$. Thus,
we obtain for all $a\in\R_{>0}$ 
\begin{align*}
f(x)-f(0)=\frac{1}{2}f''(0)x^{2}+\frac{1}{6}f'''(\xi)x^{3}\geq-\dd\rho^{a}.
\end{align*}
Therefore, also $\frac{1}{2}f''(0)x^{2}+\frac{1}{6}Mx^{3}\geq-\dd\rho^{a}$.
In this inequality we can set $x=\diff{\rho}^{a/3}$, assuming that
$a>A$ and $\diff{\rho}^{A}<r$. We get $f''(0)\ge-\left(2+\frac{M}{3}\right)\diff{\rho}^{a/3}$,
and the conclusion follows from Lem.~\ref{lem:permSignLimit} as
$a\to+\infty$.

Now assume that $f'(0)=0$ and $f''(0)>0$, so that $f''(0)>\dd\rho^{a}$
for some $a\in\R_{>0}$ by Lem.~\ref{lem:mayer}. Since $f'(0)=0$,
for all $x\in B_{r}(0)$, Taylor's formula gives
\[
f(x)-f(0)=\frac{1}{2}f''(0)x^{2}+\frac{1}{6}f'''(\xi_{x})x^{3},
\]
where $\xi_{x}\in[0,x]$. Therefore, $f(x)-f(0)>x^{2}\left(\frac{1}{2}\dd\rho^{a}+\frac{1}{6}f'''(\xi_{x})x\right)$.
Now
\[
\left|\frac{1}{6}f'''(\xi_{x})x\right|\le\frac{1}{6}M|x|\to0\ \text{as }x\to0.
\]
Thus
\[
\exists s\in\rc_{>0}:\ s<r,\ \forall x\in B_{s}(0):\ -\frac{1}{4}\dd\rho^{a}<\frac{1}{6}f'''(\xi_{x})x<\frac{1}{4}\dd\rho^{a}.
\]
We can hence write $f(x)-f(0)>x^{2}\left(\frac{1}{2}\dd\rho^{a}-\frac{1}{4}\dd\rho^{a}\right)=x^{2}\frac{1}{4}\dd\rho^{a}\ge0$
for all $x\in B_{s}(0)$, which proves that $x=0$ is a local minimum.
\end{proof}
\noindent For the generalization of Lem.~\ref{lem:local_min_diff}
and Lem.~\ref{lem:second_deriv_test} to the multivariate case, one
can proceed as above, using the ideas of \cite{KKO:08}. Note,
however, that we do not need this generalization in the present work.

\section{\label{sec:First-variation}First variation and critical points}

In this section, we define the first variation of a functional and
prove that some classical results have their counterparts in this
generalized setting, for example the fundamental lemma (Lem\@.~\ref{lem:fund_lem_calc_var})
or the connection between critical points and the Euler-Lagrange equations
(Thm.~\ref{thm:ELE}).
\begin{defn}
\label{def:gsf_0}If $a$, $b\in\rc$ and $a<b$, we define
\[
\gsf_{0}(a,b):=\left\{ \eta\in\gsf(\rc,\rc^{d})\colon\eta(a)=0=\eta(b)\right\} .
\]
When the use of the points $a$, $b$ is clear from the context, we
adopt the simplified notation $\gsf_{0}$. We also note here that
$\gsf_{0}(a,b)$ is an $\rc$-module.

\noindent One of the positive features of the use of GSF for the calculus
of variations is their closure with respect to composition. For this
reason, the next definition of functional is formally equal to the
classical one, though it can be applied to arbitrary generalized functions
$F$ and $u$.
\end{defn}
\begin{thm}
Let $a$, $b\in\rc$ with $a<b$. Let $u\in\gsf([a,b],\rc^{d})$ and
let $F\in\gsf([a,b]\times\rc^{d}\times\rc^{d},\rc)$ and define 
\begin{align}
I(u):=\int_{a}^{b}F(t,u,\dot{u})\,\dd t.\label{eq:def_I}
\end{align}
Let $\eta\in\gsf_{0}$, then
\[
\delta I(u;\eta):=\left.\der{s}I(u+s\eta)\right|_{s=0}=\int_{a}^{b}\eta\left(F_{u}(t,u,\dot{u})-\der{t}F_{\dot{u}}(t,u,\dot{u})\right)\,\dd t.
\]
\end{thm}
\begin{proof}
We have (we use Thm.~\ref{thm:rulesDer}, Thm.~\ref{thm:intRules}
and Lemma \ref{lem:int}) 
\begin{align*}
\left.\der{s}I(u+s\eta)\right|_{s=0} & =\left.\der{s}\int_{a}^{b}F(t,u+s\eta,\dot{u}+s\dot{\eta})\,\dd t\right|_{s=0}\\
 & =\int_{a}^{b}\left.\pdd{s}F(t,u+s\eta,\dot{u}+s\dot{\eta})\right|_{s=0}\,\dd t\\
 & =\int_{a}^{b}\eta F_{u}(t,u,\dot{u})+\dot{\eta}F_{\dot{u}}(t,u,\dot{u})\,\dd t\\
 & =\left[\eta F_{\dot{u}}(t,u,\dot{u})\right]_{a}^{b}+\int_{a}^{b}\eta\left(F_{u}(t,u,\dot{u})-\der{t}F_{\dot{u}}(t,u,\dot{u})\right)\,\dd t\\
 & =\int_{a}^{b}\eta\left(F_{u}(t,u,\dot{u})-\der{t}F_{\dot{u}}(t,u,\dot{u})\right)\,\dd t.
\end{align*}
\end{proof}
We call $\delta I(u;\eta)$ the \emph{first variation} of $I$. In
addition we call $u\in\gsf([a,b],\rc^{d})$ a \emph{critical point}
of $I$ if $\delta I(u;\eta)=0$ for all $\eta\in\gsf_{0}$.

To prove the fundamental lemma of the calculus of variations, Lem.~\ref{lem:fund_lem_calc_var},
we first show that every GSF can be approximated using generalized
strict delta nets. 
\begin{lem}
\label{lem:limStrictDeltaNet}Let $a$, $b\in\rc$ be such that $a<b$
and let $f\in\gsf([a,b],\rc)$. Let $x\in[a,b]$ and $R\in\rc_{>0}$
be such that $B_{R}(x)\subseteq[a,b]$. Assume that $G_{t}\in\gsf(\rc,\rc)$
satisfy

\begin{enumerate}
\item \label{enu:int1}$\int_{-R}^{R}G_{t}=1$ for $t\in\rc_{>0}$ small.
\item \label{enu:G_tZero}For $t$ small, $(G_{t})_{t\in\rc_{>0}}$ is zero
outside every ball $B_{\delta}(0)$, $0<\delta<R$, i.e.
\begin{equation}
\forall\delta\in\rc_{>0}\,\exists\rho\in\rc_{>0}\,\forall t\in B_{\rho}(0)\cap\rc_{>0}\,\forall y\in[-R,-\delta]\cup[\delta,R]:\ G_{t}(y)=0.\label{eq:G_tIsZero}
\end{equation}
\item \label{enu:delta_bd} $\exists M\in \rc_{>0}\,\exists \rho\in\rc \,\forall t\in B_{\rho}(0)\colon \int_{-R}^R\left|G_t(y)\right|\,\dd y \leq M $.
\end{enumerate}
Then
\[
\lim_{t\to0^{+}}\int_{-R}^{R}f(x-y)G_{t}(y)\diff{y}=f(x).
\]
Moreover $\int_{-R}^{R}f(x-y)G_{t}(y)\diff{y}=\int_{x-R}^{x+R}f(y)G_{t}(x-y)\diff{y}$.

\end{lem}
\begin{proof}
We only have to generalize the classical proof concerning limits of
convolutions with strict delta nets. We first note that 
\[
\int_{-R}^{R}f(x-y)G_{t}(y)\diff{y}=\int_{x-R}^{x+R}f(y)G_{t}(x-y)\diff{y}
\]
so that these integrals exist because $(x-R,x+R)=B_{R}(x)\subseteq[a,b]$.
Using \ref{enu:int1}, for $t$ small, let's say for $0<t<S\in\rc_{>0}$,
we get
\begin{align*}
\left|\int_{-R}^{R}f(x-y)G_{t}(y)\diff{y}-f(x)\right| & =\left|\int_{-R}^{R}\left[f(x-y)-f(x)\right]G_{t}(y)\diff{y}\right|\\
 & \le\int_{-R}^{R}\left|f(x-y)-f(x)\right|\cdot\left|G_{t}(y)\right|\diff{y}.
\end{align*}
For each $r\in\rc_{>0}$, sharp continuity of $f$ at $x$ yields
$\left|f(x-y)-f(x)\right|<r$ for all $y$ such that $|y|<\delta\in\rc_{>0}$,
and we can take $\delta<R$. By \ref{enu:G_tZero}, for $0<|t|<\min(\rho,S)$,
we have
\begin{equation}
\left|\int_{-R}^{R}f(x-y)G_{t}(y)\diff{y}-f(x)\right|\le r\int_{-\delta}^{+\delta}\left|G_{t}(y)\right|\diff{y}.\label{eq:rTimesIntegG}
\end{equation}
The right hand side of \eqref{eq:rTimesIntegG} can be taken arbitrarily
small in $\rc_{>0}$ because $[-\delta,\delta]\fcmp\rc$, \ref{enu:delta_bd} and because
of the extreme value theorem \ref{thm:extremeValues} applied to the
GSF $G_{t}$.
\end{proof}
\begin{lem}[Fundamental Lemma of the Calculus of Variations]
\label{lem:fund_lem_calc_var}Let $a$, $b\in\rc$ such that $a<b$,
and let $f\in\gsf([a,b],\rc)$. If
\begin{equation}
\int_{a}^{b}f(t)\eta(t)\,\dd t=0\:\text{ for all }\:\eta\in\gsf_{0},\label{eq:HpFundLem}
\end{equation}
then $f=0$.
\end{lem}
\begin{proof}
Let $x\in[a,b]$. Because of Thm.~\ref{thm:propGSF}.\ref{enu:GSF-cont}
and Lem.~\ref{lem:approxOfBoundaryPointsWithInterior}, without loss
of generality we can assume that $x$ is a sharply interior point,
so that $B_{R}(x)\subseteq[a,b]$ for some $R\in\rc_{>0}$. Let $\phi\in\D_{[-1,1]}(\R)$
be such that $\int\phi=1$. Set $G_{t,\eps}(x):=\frac{1}{t_{\eps}}\phi\left(\frac{x}{t_{\eps}}\right)$,
where $x\in\R$ and $t\in\rc_{>0}$, and $G_{t}(x):=[G_{t,\eps}(x_{\eps})]$
for all $x\in\rc$. Then, for $t$ sufficiently small, we have $G_{t}(x-.)\in\gsf_{0}$
and \eqref{eq:HpFundLem} yields $\int_{a}^{b}f(y)G_{t}(x-y)\diff{y}=0$.
For $t$ small, we both have that $G_{t}(x-.)=0$ on $[a,x-R]\cup[x+R,b]$
and the assumptions of Lem.~\ref{lem:limStrictDeltaNet} hold. Therefore
\begin{align*}
0 & =\int_{a}^{b}f(y)G_{t}(x-y)\diff{y}=\int_{x-R}^{x+R}f(y)G_{t}(x-y)\diff{y}=\\
 & =\int_{-R}^{R}f(x-y)G_{t}(y)\diff{y},
\end{align*}
and Lem.~\ref{lem:limStrictDeltaNet} hence yields $f(x)=0$.
\end{proof}
Thus we obtain the following
\begin{thm}
\label{thm:ELE}Let $a$, $b\in\rc$ such that $a<b$, and let $u\in\gsf([a,b],\rc^{d})$.
Then u solves the Euler-Lagrange equations 
\begin{align}
F_{u}-\der{t}F_{\dot{u}}=0\label{eq:ELE}
\end{align}
for $I$ given by \eqref{eq:def_I}, if and only if $\delta I(u;\eta)=0$
for all $\eta\in\gsf_{0}$, i.e.~if and only if $u$ is a critical
point of $I$.
\end{thm}

\section{\label{sec:second-variation}second variation and minimizers}

As in the classical case (see e.g.~\cite{GeFo00}), thanks to the
extreme value theorem \ref{thm:extremeValues} and the property of
the interval $[a,b]$ of being functionally compact, we can naturally
define a topology on the space $\gsf([a,b],\rc^{d})$:
\begin{defn}
\label{def:genNormsSpaceGSF}Let $a$, $b\in\rc$, with $a<b$. Let
$m\in\N_0$ and $v\in\gsf([a,b],\rc^{d})$. Then
\[
\Vert v\Vert_{m}:=\max_{\substack{n\le m\\
1\le i\le d
}
}\max\left(\left|\frac{d^{n}}{dt^{n}}v^{i}(M_{ni})\right|,\left|\frac{d^{n}}{dt^{n}}v^{i}(m_{ni})\right|\right)\in\rc,
\]
where $M_{ni}$, $m_{ni}\in[a,b]$ satisfy
\[
\forall t\in[a,b]:\ \frac{d^{n}}{dt^{n}}v^{i}(m_{ni})\le\frac{d^{n}}{dt^{n}}v^{i}(t)\le\frac{d^{n}}{dt^{n}}v^{i}(M_{ni}).
\]
\end{defn}
\noindent The following result permits to calculate the (generalized)
norm $\Vert v\Vert_{m}$ using any net $(v_{\eps})$ that defines
$v$.
\begin{lem}
\label{lem:normSpaceGSF}Under the assumptions of Def.~\ref{def:genNormsSpaceGSF},
let $a=[a_{\eps}]$ and $b=[b_{\eps}]$ be such that $a_{\eps}<b_{\eps}$
for all $\eps$. Then we have:

\begin{enumerate}
\item \label{enu:normAndDefNet}If the net $(v_{\eps})$ defines $v$, then
$\Vert v\Vert_{m}=\left[\max_{\substack{n\le m\\
1\le i\le d
}
}\max_{t\in[a_{\eps},b_{\eps}]}\left|\frac{d^{n}}{dt^{n}}v_{\eps}^{i}(t)\right|\right]$;
\item \label{enu:normPos}$\Vert v\Vert_{m}\ge0$;
\item $\Vert v\Vert_{m}=0$ if and only if $v=0$;
\item $\forall c\in\rc:\ \Vert c\cdot v\Vert_{m}=|c|\cdot\Vert v\Vert_{m}$;
\item \label{enu:normTriang}For all $u\in\gsf([a,b],\rc^{d})$, we have
$\Vert u+v\Vert_{m}\le\Vert u\Vert_{m}+\Vert v\Vert_{m}$ and $\Vert u\cdot v\Vert_{m}\le c_{m}\cdot\Vert u\Vert_{m}\cdot\Vert v\Vert_{m}$
for some $c_{m}\in\rc_{>0}$.
\end{enumerate}
\end{lem}
\begin{proof}
By the standard extreme value theorem applied $\eps$-wise, we get
the existence of $\bar{m}_{ni\eps}$, $\bar{M}_{ni\eps}\in[a_{\eps},b_{\eps}]$
such that
\[
\forall t\in[a_{\eps},b_{\eps}]:\ \frac{d^{n}}{dt^{n}}v_{\eps}^{i}(\bar{m}_{ni\eps})\le\frac{d^{n}}{dt^{n}}v_{\eps}^{i}(t)\le\frac{d^{n}}{dt^{n}}v_{\eps}^{i}(\bar{M}_{ni\eps}).
\]
Hence $\left|\frac{d^{n}}{dt^{n}}v_{\eps}^{i}(t)\right|\le\max\left(\left|\frac{d^{n}}{dt^{n}}v_{\eps}^{i}(\bar{m}_{ni\eps})\right|,\left|\frac{d^{n}}{dt^{n}}v_{\eps}^{i}(\bar{M}_{ni\eps})\right|\right)$.
Thus
\[
\max_{\substack{n\le m\\
1\le i\le d
}
}\max_{t\in[a_{\eps},b_{\eps}]}\left|\frac{d^{n}}{dt^{n}}v_{\eps}^{i}(t)\right|\le\max_{\substack{n\le m\\
1\le i\le d
}
}\max\left(\left|\frac{d^{n}}{dt^{n}}v_{\eps}^{i}(\bar{m}_{ni\eps})\right|,\left|\frac{d^{n}}{dt^{n}}v_{\eps}^{i}(\bar{M}_{ni\eps})\right|\right).
\]
But $\bar{m}_{ni\eps}$, $\bar{M}_{ni\eps}\in[a_{\eps},b_{\eps}]$,
so
\begin{align*}
\left[\max_{\substack{n\le m\\
1\le i\le d
}
}\max_{t\in[a_{\eps},b_{\eps}]}\left|\frac{d^{n}}{dt^{n}}v_{\eps}^{i}(t)\right|\right] & =\left[\max_{\substack{n\le m\\
1\le i\le d
}
}\max\left(\left|\frac{d^{n}}{dt^{n}}v_{\eps}^{i}(\bar{m}_{ni\eps})\right|,\left|\frac{d^{n}}{dt^{n}}v_{\eps}^{i}(\bar{M}_{ni\eps})\right|\right)\right]=\\
 & =\max_{\substack{n\le m\\
1\le i\le d
}
}\max\left(\left|\frac{d^{n}}{dt^{n}}v^{i}(\bar{m}_{ni})\right|,\left|\frac{d^{n}}{dt^{n}}v^{i}(\bar{M}_{ni})\right|\right).
\end{align*}
This proves both that $\Vert v\Vert_{m}$ is well-defined, i.e.~it
does not depend on the particular choice of points $m_{ni}$, $M_{ni}$
as in Def.~\ref{def:genNormsSpaceGSF}, and the claim \ref{enu:normAndDefNet}.
The remaining properties \ref{enu:normPos} - \ref{enu:normTriang}
follows directly from \ref{enu:normAndDefNet} and the usual properties
of standard $\mathcal{C}^{m}$-norms.
\end{proof}
Using these $\rc$-valued norms, we can naturally define a topology
on the space $\gsf([a,b],\rc^{d})$.
\begin{defn}
\label{def:sharpTopSpaceGSF}Let $a$, $b\in\rc$, with $a<b$. Let
$m\in\N$, $u\in\gsf([a,b],\rc^{d})$, $r\in\rc_{>0}$, then

\begin{enumerate}
\item $B_{r}^{m}(u):=\left\{ v\in\gsf([a,b],\rc^{d})\mid\Vert v-u\Vert_{m}<r\right\} $
\item If $U\subseteq\gsf([a,b],\rc^{d})$, then we say that $U$ is a \emph{sharply
open set} if
\[
\forall u\in U\,\exists m\in\N\,\exists r\in\rc_{>0}:\ B_{r}^{m}(u)\subseteq U.
\]
\end{enumerate}
\end{defn}
\noindent As in \cite[Thm.~2]{GKV}, one can easily prove that sharply
open sets form a topology on $\gsf([a,b],\rc^{d})$. Using this topology,
we can define when a curve is a minimizer of the functional $I$.
Note explicitly that there are no restrictions on the generalized
numbers $a$, $b\in\rc$, $a<b$. E.g.~they can also both be infinite.
\begin{defn}
\label{def:minimizer}Let $a$, $b\in\rc$, with $a<b$ and let $u\in\gsf([a,b],\rc^{d})$,
then

\begin{enumerate}
\item For all $p$, $q\in\rc^{d}$, we set
\[
\gsf_{\text{bd}}(p,q):=\left\{ v\in\gsf([a,b],\rc^{d})\mid v(a)=p,\ v(b)=q\right\} .
\]
Note that $\gsf_{\text{bd}}(0,0)=\gsf_{0}$. The subscript ``bd'' stands here for ``boundary values''.
\item We say that $u$ is a \emph{local minimizer of }$I$ \emph{in} $\gsf_{\text{bd}}(p,q)$
if $u\in\gsf_{\text{bd}}(p,q)$ and
\begin{equation}
\exists r\in\rc_{>0}\,\exists m\in\N\,\forall v\in B_{r}^{m}(u)\cap\gsf_{\text{bd}}(p,q):\ I(v)\ge I(u)\label{eq:defMinimizer}
\end{equation}
\item We define the \emph{second variation} of $I$ in direction $\eta\in\gsf_{0}$
as
\[
\delta^{2}I(u;\eta):=\left.\difff{s}\right|_{0}I(u+s\eta).
\]
\end{enumerate}
\end{defn}
\noindent Note also explicitly that the points $p$, $q\in\rc^{d}$
can have infinite norm, e.g.~$|p_{\eps}|\to+\infty$ as $\eps\to0$.
We calculate, by using the standard Einstein's summation conventions
\begin{align*}
\delta^{2}I(u;\eta) & =\left.\difff{s}\right|_{0}\int_{a}^{b}F(t,u+s\eta,\dot{u}+s\dot{\eta})\,\dd t\\
 & =\int_{a}^{b}\left.\pdddp{s}\right|_{0}F(t,u+s\eta,\dot{u}+s\dot{\eta})\,\dd t\\
 & =\int_{a}^{b}F_{u^{i}u^{j}}(t,u,\dot{u})\eta^{i}\eta^{j}+2F_{u^{i}\dot{u}^{j}}(t,u,\dot{u})\eta^{i}\dot{\eta}^{j}+F_{\dot{u}^{i}\dot{u}^{j}}(t,u,\dot{u})\dot{\eta}^{i}\dot{\eta}^{j}\,\dd t,
\end{align*}
which we abbreviate as
\[
\delta^{2}I(u;\eta)=\int_{a}^{b}F_{uu}(t,u,\dot{u})\eta\eta+2F_{u\dot{u}}(t,u,\dot{u})\eta\dot{\eta}+F_{\dot{u}\dot{u}}(t,u,\dot{u})\dot{\eta}\dot{\eta}\,\dd t.
\]

The following results establish classical necessary and sufficient
conditions to decide if a function $u$ is a minimizer for the given
functional \eqref{eq:def_I}. 
\begin{thm}
\label{thm:necessCondsForMinimizer}Let $a$, $b\in\rc$ with $a<b$,
let $F\in\gsf([a,b]\times\rc^{d}\times\rc^{d},\rc)$, let $p$, $q\in\rc^{d}$
and let $u$ be a local minimizer of $I$ in $\gsf_{\text{\emph{bd}}}(p,q)$.
Then

\begin{enumerate}
\item $\delta I(u;\eta)=0$ for all $\eta\in\gsf_{0}$;
\item \label{enu:2ndVarPosNec}$\delta^{2}I(u;\eta)\geq0$ for all $\eta\in\gsf_{0}$.
\end{enumerate}
\end{thm}
\begin{proof}
Let $r\in\rc_{>0}$ be such that \eqref{eq:defMinimizer} holds. Since
$\eta\in\gsf_{0}$, the map $s\in\rc\mapsto u+s\eta\in\gsf_{\text{bd}}(p,q)$
is well defined and continuous with respect to the trace of the sharp
topology in its codomain. Therefore, we can find $\bar{r}\in\rc_{>0}$
such that $u+s\eta\in B_{r}^{m}(u)\cap\gsf_{\text{bd}}(p,q)$ for
all $s\in B_{\bar{r}}(0)$. We hence have $I(u+s\eta)\ge I(u)$. This
shows that the GSF $s\in B_{\bar{r}}(0)\mapsto I(u+s\eta)\in\rc$
has a local minimum at $s=0$. Now, we employ Lem.~\ref{lem:local_min_diff}
and Lem.~\ref{lem:second_deriv_test} and thus the claims are proven. 
\end{proof}
\begin{thm}
\label{thm:suffCondsForMinimizer}Let $a$, $b\in\rc$ with $a<b$
and $p$, $q\in\rc^{d}$. Let $u\in\gsf_{\text{\emph{bd}}}(p,q)$
be such that 

\begin{enumerate}
\item \label{enu:1stVarZero}$\delta I(u;\eta)=0$ for all $\eta\in\gsf_{0}$. 
\item \label{enu:2ndVarZero}$\delta^{2}I(v;\eta)\geq0$ for all $\eta\in\gsf_{0}$
and for all $v\in B_{r}^{m}(u)\cap\mathcal{GC}_{\text{\emph{bd}}}^{\infty}\left(p,q\right)$,
where $r\in\rc_{>0}$ and $m\in\N$.
\end{enumerate}
Then $u$ is a local minimizer of the functional $I$ in $\gsf_{\text{\emph{bd}}}(p,q)$. 

Moreover, if $\delta^{2}I(v;\eta)>0$ for all $\eta\in\gsf_{0}$ such that $\Vert\eta\Vert_{m}>0$
and for all $v\in B_{2r}^{m}(u)\cap\mathcal{GC}_{\text{\emph{bd}}}^{\infty}\left(p,q\right)$,
then $I(v)>I(u)$ for all $v\in B_{r}^{m}(u)\cap\mathcal{GC}_{\text{\emph{bd}}}^{\infty}\left(p,q\right)$
such that $\Vert v-u\Vert_{m}>0$.

\end{thm}
\begin{proof}
For any $v\in B_{r}^{m}(u)\cap\mathcal{GC}_{\text{bd}}^{\infty}\left(p,q\right)$,
we set $\psi(s):=I(u+s(v-u))\in\rc$ for all $s\in B_{1}(0)$ so
that $u+s(v-u)\in B_{r}^{m}(u)$. Since $(v-u)(a)=0=(v-u)(b)$, we
have $v-u\in\gsf_{0}$, and properties \ref{enu:1stVarZero}, \ref{enu:2ndVarZero}
yield $\psi'(0)=\delta I(u;v-u)=0$ and $\psi''(s)=\delta^{2}I(u+s(v-u);v-u)\ge0$
for all $s\in B_{1}(0)$. We claim that $s=0$ is a minimum of $\psi$.
In fact, for all $s\in B_{1}(0)$ by Taylor's theorem \ref{thm:Taylor}
\[
\psi(s)=\psi(0)+s\psi'(0)+\frac{s^{2}}{2}\psi''(\xi)
\]
for some $\xi\in[0,s]$. But $\psi'(0)=0$ and hence $\psi(s)-\psi(0)=\frac{s^{2}}{2}\psi''(\xi)\ge0$.
Finally, Lem.~\ref{lem:permSignLimit} yields
\[
\lim_{s\to1^{-}}\psi(s)=I(v)\ge\psi(0)=I(u),
\]
which is our conclusion. Note explicitly that if $\delta^{2}I(v;\eta)=0$
for all $\eta\in\gsf_{0}$ and for all $v\in B_{r}^{m}(u)\cap\mathcal{GC}_{\text{bd}}^{\infty}\left(p,q\right)$,
then $\psi''(\xi)=0$ and hence $I(v)=I(u)$.

Now, assume that $\delta^{2}I(v;\eta)>0$ for all $\eta\in\gsf_{0}$
such that $\Vert\eta\Vert_{m}>0$ and for all $v\in B_{2r}^{m}(u)\cap\mathcal{GC}_{\text{bd}}^{\infty}\left(p,q\right)$,
and take $v\in B_{r}^{m}(u)\cap\mathcal{GC}_{\text{bd}}^{\infty}\left(p,q\right)$
such that $\Vert v-u\Vert_{m}>0$. As above set $\psi(s):=I(u+s(v-u))\in\rc$
for all $s\in B_{3/2}(0)$ so that $u+s(v-u)\in B_{2r}^{m}(u)$.
We have $\psi'(0)=0$ and $\psi''(s)=\delta^{2}I(u+s(v-u);v-u)>0$
for all $s\in B_{3/2}(0)$ because $\Vert v-u\Vert_{m}>0$. Using
Taylor's theorem, we get $\psi(1)=\psi(0)+\frac{1}{2}\psi''(\xi)$
for some $\xi\in[0,1]$. Therefore $\psi(1)-\psi(0)=I(v)-I(u)=\frac{1}{2}\psi''(\xi)>0$.
\end{proof}
\begin{lem}
\label{lem:mvt}Let $(a_{k})_{k\in\N}$, $(b_{k})_{k\in\N}$ and $(c_{k})_{k\in\N}$
be sequences in $\rc_{>0}$. Assume that both $(a_{k})_{k}$, $(b_{k})_{k}\to0$
and $\frac{c_{k}}{a_{k}+b_{k}}\to1$ in the sharp topology as $k\to+\infty$.
Let $f\in\gsf([a_{1},b_{1}],\rc)$. Finally, let $a_{k}<t<b_{k}$
for all $k\in\N$, then it holds that
\[
f(t)=\lim_{k\to\infty}\frac{1}{c_{k}}\int_{t-a_{k}}^{t+b_{k}}f(s)\,\dd s.
\]
\end{lem}
\begin{proof}
We can apply the integral mean value theorem for each $\eps$ and each defining
net $(f_{\eps})$ of $f$ to get the existence of $\tau_{k}\in[t-a_{k},t+b_{k}]$
such that
\begin{align*}
f(\tau_{k}) & =\frac{1}{b_{k}+a_{k}}\int_{t-a_{k}}^{t+b_{k}}f(s)\,\dd s\\
 & =\frac{c_{k}}{b_{k}+a_{k}}\frac{1}{c_{k}}\int_{t-a_{k}}^{t+b_{k}}f(s)\,\dd s.
\end{align*}
Now, we take the limit for $k\to\infty$, and the claim follows by
assumption and by Thm\@.~\ref{thm:propGSF}.\ref{enu:GSF-cont},
i.e.~by sharp continuity of $f$.
\end{proof}
We now derive the so-called necessary Legendre condition:
\begin{thm}
\label{Thm:LH}Let $a$, $b\in\rc$ with $a<b$ and let $u\in\gsf([a,b],\rc^{d})$
be a minimizer of the functional $I$. Then
\[
F_{\dot{u}\dot{u}}(t,u(t),\dot{u}(t))
\]
is positive semi definite for all $t\in[a,b]$, i.e.
\begin{equation}
F_{\dot{u}^{i}\dot{u}^{j}}(t,u(t),\dot{u}(t))\lambda^{i}\lambda^{j}\geq0,\quad\forall\lambda=(\lambda^{1},\ldots,\lambda^{d})\in\rc^{d}.\label{eq:LH}
\end{equation}
\end{thm}
\begin{proof}
Let $\lambda=[\lambda_{\eps}]\in\rc^{d}$ and $k$, $h\in\N$ be arbitrary.
Let $t=[t_{\varepsilon}]\in[a,b]$. We can assume that $t$ is a sharply
interior point, because otherwise we can use sharp continuity of the
left hand side of \eqref{eq:LH} and Lem.~\ref{lem:permSignLimit}.
We can also assume that $\lambda$ is componentwise invertible because
of Lem.~\ref{lem:componentWisePos}. We want to mimic the classical
proof of \cite[Thm. 1.3.2]{JoJo98}, but considering a ``regularized''
version of the triangular function used there (see Fig.~\ref{fig:lp}).
In particular: (1) the smoothed triangle must have an infinitesimal
height which is proportional to $\lambda$, and we will take $\rho_{\eps}^{k}$
as this infinitesimal; (2) in the proof we need that the derivative
at $t$ is equal to $\lambda$, and this justifies the drawing of
the peak in Fig.~\ref{fig:lp}; (3) to regularize the singular points
of the triangular function, we need a smaller infinitesimal, and we
can take e.g.~$\rho_{\eps}^{2k}$. So, consider a net of smooth functions
$\vartheta_{\varepsilon}$ on $[a_{\ve},b_{\ve}]$, such that the
following properties hold:

\begin{enumerate}
\item \label{enu:zero1}$\vartheta_{\ve}(x)=0$, for $x\leq t_{\varepsilon}-\rho_{\varepsilon}^{k}-\rho_{\varepsilon}^{2k}$. 
\item \label{enu:zero2}$\vartheta_{\ve}(x)=0$, for $x\geq t_{\varepsilon}+\rho_{\varepsilon}^{k}+\rho_{\varepsilon}^{2k}$. 
\item \label{enu:varthetaInT_eps}$\vartheta_{\ve}(x)=\lambda(x-t_{\eps})+\rho_{\eps}^{k}\lambda$
for $x\in[t_{\varepsilon}-\rho_{\varepsilon}^{k}+\rho_{\varepsilon}^{2k},t_{\varepsilon}]$.
\item $\vartheta_{\ve}(x)=-\lambda(x-t_{\eps})+\rho_{\eps}^{k}\lambda$
for $x\in[t_{\ve}+\rho_{\ve}^{2k},t_{\ve}+\rho_{\ve}^{k}-\rho_{\ve}^{2k}]$.
\item \label{enu:varthetaBound}$|\vartheta_{\eps}(x)|\le\rho_{\eps}^{k}\cdot|\lambda|+2\rho_{\eps}^{2k}|\lambda|$.
\item \label{enu:varthetaDotBound}$|\dot{\vartheta}_{\eps}(x)|\le2|\lambda|$
for all $x$
\end{enumerate}
\begin{figure}[h!]
\begin{center}
\def\svgwidth{0.8\textwidth} 
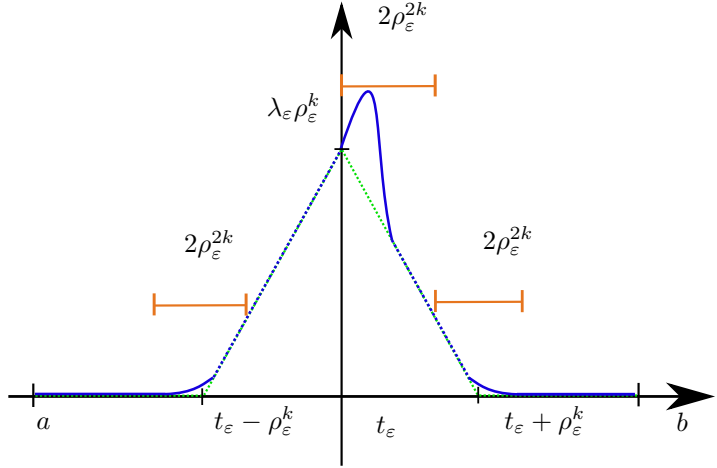
\caption{This figure illustrates the function $\vartheta_{\ve}$ we are considering
(\textcolor{blue}{blue}). The dotted \textcolor{green}{green} triangle symbolizes the function which is used in the classical proofs of the
Legendre necessary condition (cf.~\cite[Thm. 1.3.2]{JoJo98}).}\label{fig:lp}
\end{center}
\end{figure}

\noindent The net $(\vartheta_{\ve})$ defines a GSF $\vartheta:=[\vartheta_{\eps}(-)]\in\gsf_{0}$
because $t$ is a sharply interior point. Setting for simplicity $a_{k}:=\dd\rho^{k}+\dd\rho^{2k}$,
by assumption we have 
\begin{align}
0\leq\delta^{2}I(u,\vartheta)=\int_{t-a_{k}}^{t+a_{k}}F_{uu}(t,u,\dot{u})\vartheta\vartheta+2F_{\dot{u}u}(t,u,\dot{u})\dot{\vartheta}\vartheta+F_{\dot{u}\dot{u}}(t,u,\dot{u})\dot{\vartheta}\dot{\vartheta}\,\dd t.\label{eq:proof_hadamard}
\end{align}
Now, setting $M:=\max_{[a,b]}|F_{uu}(t,u,\dot{u})|$ and $N:=\max_{[a,b]}|F_{u\dot{u}}(t,u,\dot{u})|$,
by \ref{enu:varthetaBound} we
\[
\left|\int_{t-a_{k}}^{t+a_{k}}F_{uu}(t,u,\dot{u})\vartheta\vartheta\,\dd t\right|\leq M\cdot|\vartheta(t)|^{2}\cdot2a_{k}=O\left(\dd\rho^{3k}\right),
\]
where we used the evident notation $G_{k}=O\left(\dd\rho^{k}\right)$
to denote that there exists some $A\in\rc_{>0}$ such that $G_{k}\le A\cdot\dd\rho^{k}$
for all $k\in\N$. Using \ref{enu:varthetaBound} and \ref{enu:varthetaDotBound},
we analogously have
\[
\left|\int_{t-a_{k}}^{t+a_{k}}F_{\dot{u}u}(t,u,\dot{u})\dot{\vartheta}\vartheta\,\dd t\right|\leq4N\cdot|\vartheta(t)|\cdot a_{k}\cdot|\lambda|=O\left(\dd\rho^{2k}\right).
\]
Note that there always exists $C\in \rc$ such that $|\lambda|\leq C\dd \rho^k$. Therefore
\begin{equation}
\lim_{k\to+\infty}\frac{1}{2\dd\rho^{k}}\int_{t-a_{k}}^{t+a_{k}}F_{uu}(t,u,\dot{u})\vartheta\vartheta+2F_{\dot{u}u}(t,u,\dot{u})\dot{\vartheta}\vartheta\,\dd t=0.\label{eq:2IntLim0}
\end{equation}
Using Lemma \ref{lem:mvt}, \eqref{eq:2IntLim0}, \eqref{eq:proof_hadamard}
and Lem.~\ref{lem:permSignLimit}, we obtain that
\[
F_{\dot{u}\dot{u}}(t,u(t)\dot{u}(t))\dot{\vartheta}(t)\dot{\vartheta}(t)=\lim_{k\to+\infty}\frac{1}{2\dd\rho^{k}}\int_{t-a_{k}}^{t+a_{k}}F_{\dot{u}\dot{u}}(t,u,\dot{u})\dot{\vartheta}\dot{\vartheta}\,\dd t\ge0.
\]
But \ref{enu:varthetaInT_eps} yields $\dot{\vartheta}(t)=\lambda$,
and this concludes the proof.
\end{proof}

\section{\label{sec:Jacobi-fields}Jacobi fields}

As in the classical case, Thm.~\ref{thm:necessCondsForMinimizer}.\ref{enu:2ndVarPosNec}
motivates to define the \emph{accessory} integral
\begin{equation}
Q(\eta):=\int_{a}^{b}\psi(t,\eta,\dot{\eta})\,\dd t\quad\forall\eta\in\gsf_{0},\label{eq:accessInt}
\end{equation}
where 
\begin{align}
\psi(t,l,v):=F_{uu}(t,u,\dot{u})ll+2F_{u\dot{u}}(t,u,\dot{u})lv+F_{\dot{u}\dot{u}}(t,u,\dot{u})vv\label{eq:aux_prob_psi}
\end{align}
for all $t\in[a,b]$ and $(l,v)\in\rc^{d}\times\rc^{d}$. Note that
if $u$ minimizes $I$, then
\[
Q(\eta)\geq0\quad\forall\eta\in\gsf_{0}.
\]

As usual, we note that $\eta=0$ is a minimizer of the functional
$Q$ and we are interested whether there are others. In order to solve
this problem, we consider the \emph{Euler-Lagrange }equations for
$Q$, which are given by
\begin{align}
\der{t}\psi_{\dot{\eta}}(t,\eta,\dot{\eta})=\psi_{\eta}(t,\eta,\dot{\eta}),
\end{align}
in other words 
\begin{align}
\der{t}\left\{ F_{\dot{u}\dot{u}}(t,u,\dot{u})\dot{\eta}+F_{u\dot{u}}(t,u,\dot{u})\eta\right\} =F_{u\dot{u}}(t,u,\dot{u})\dot{\eta}+F_{uu}(t,u,\dot{u})\eta.\label{eq:jacobi}
\end{align}
Since $u$ is given, \eqref{eq:jacobi} is an $\rc$-linear system
of second order equations in the unknown GSF $\eta$ and with time
dependent coefficients in $\rc$. We call \eqref{eq:jacobi} the \emph{Jacobi
equations} \emph{of }$I$ \emph{with respect to} $u$. As in the classical
setting, we define 
\begin{defn}
A solution $\eta\in\gsf_{0}$ of the Jacobi equations \eqref{eq:jacobi}
is called a \emph{Jacobi field along} $u_0=u$.
\end{defn}
The following result confirms that the intuitive interpretation of
a Jacobi field as the tangent space of a smooth family of solutions
of the Euler-Lagrange equation still holds in this generalized setting.
\begin{lem}
Let $u\in\gsf([-\delta,\delta]\times[a,b],\rc^{d})$, where $\delta\in\rc_{>0}$.
We write $u_{s}:=u(s,-)$ for all $s\in[-\delta,\delta]$. Assume
that each $u_{s}$ satisfies the Euler-Lagrange equations \eqref{eq:ELE}:
\[
\der{t}F_{\dot{u}}(t,u_{s},\dot{u}_{s})=F_{u}(t,u_{s},\dot{u}_{s})\quad\forall s\in[-\delta,\delta].
\]
Then
\[
\eta(t):=\left.\der{s}\right|_{0}u_{s}(t)\quad\forall t\in[a,b]
\]
is a Jacobi field along $u$.
\end{lem}
\begin{proof}
A straight forward calculation gives:
\begin{align*}
0 & =\left.\der{s}\right|_{0}\left(\der{t}F_{\dot{u}}(t,u_{s},\dot{u}_{s})-F_{u}(t,u_{s},\dot{u}_{s})\right)\\
 & =\der{t}\left(F_{\dot{u}\dot{u}}(t,u,\dot{u})\dot{\eta}+F_{u\dot{u}}(t,u,\dot{u})\eta\right)-F_{u\dot{u}}(t,u,\dot{u})\dot{\eta}-F_{uu}(t,u,\dot{u})\eta.
\end{align*}
\end{proof}

\subsection{Conjugate points and Jacobi's theorem}

The classical key result concerning Jacobi fields relates
conjugate points and minimizers. The main aim of the present section
is to derive this theorem in our generalized framework by extending
the ideas of the proof of \cite[Thm.~1.3.4]{JoJo98}.

A crucial notion is hence that of piecewise GSF:
\begin{defn}
We call \emph{piecewise GSF} an $n$-tuple $(f_{1},\ldots,f_{n})$
such that

\begin{enumerate}
\item \label{enu:tupleOfGSF}For all $i=1,\ldots,n$ there exist $a_{i}$,
$a_{i+1}\in\rc$ such that $a_{i}<a_{i+1}$ and $f_{i}\in\gsf([a_{i},a_{i+1}],\rc^{d})$.
Note that $[a,b]=[a',b']$ implies $a=a'$ and $b=b'$ because the
relation $\le$ is antisymmetric. Therefore, the points $a_{i}$,
$a_{i+1}$ are uniquely determined by the set-theoretical function
$f_{i}$.
\item \label{enu:evaluationWellDef}For all $i=1,\ldots,n$, we have $f_{i}(a_{i+1})=f_{i+1}(a_{i+1})$.
\end{enumerate}
Every pointwise GSF $(f_{1},\ldots,f_{n})$ defines a set-theoretical
function:
\begin{enumerate}[resume]
\item \label{enu:evaluationPGSF}For all $t\in\bigcup_{i=1}^{n}[a_{i},a_{i+1}]$,
we set $(f_{1},\ldots,f_{n})(t):=f_{i}(t)$ if $t\in[a_{i},a_{i+1}]$.
\end{enumerate}
We also use the arrow notation $(f_{1},\ldots,f_{n}):\bigcup_{i=1}^{n}[a_{i},a_{i+1}]\ra\rc^{d}$
to say that both \ref{enu:tupleOfGSF} and \ref{enu:evaluationWellDef}
hold.

\end{defn}
\begin{rem}
~

\begin{enumerate}
\item Clearly, $t\in[a_{i},a_{i+1}]\cap[a_{i+1},a_{i+2}]$ implies $t=a_{i+1}$,
so that condition \ref{enu:evaluationWellDef} yields that the evaluation
\ref{enu:evaluationPGSF} is well defined.
\item Since the order relation $\le$ is not a total one, we do not have
that $[a_{i},a_{i+1}]\cup[a_{i+1},a_{i+2}]=[a_{i},a_{i+2}]$.
\item If $\nu:[a_{1},a_{2}]\cup[a_{2},a_{3}]\ra\rc^{d}$ is a set-theoretical
function originating from a piecewise GSF $(f_{1},f_{2})$, then neither
the GSF $f_{i}$ nor the points $a_{i}$ are uniquely determined by
$\nu$. For this reason, we prefer to stress our notations with symbols
like $(f_{1},f_{2})(t)\in\rc^{d}$.
\item Every GSF $f\in\gsf([a_{1},a_{2}],\rc^{d})$ can be seen as a particular
case of a piecewise GSF.
\item If $(g_{1},\ldots,g_{n})$, $(f_{1},\ldots,f_{n}):\bigcup_{i=1}^{n}[a_{i},a_{i+1}]\ra\rc^{d}$
and $r\in\rc$, then also $(g_{1},\ldots,g_{n})+(f_{1},\ldots,f_{n}):=(g_{1}+f_{1},\ldots,g_{n}+f_{n})$
and $r\cdot(f_{1},\ldots,f_{n}):=(r\cdot f_{1},\ldots,r\cdot f_{n})$
are piecewise GSF, and we hence have a structure of $\rc$-module.
\item If $(f_{1},\ldots,f_{n}):\bigcup_{i=1}^{n}[a_{i},a_{i+1}]\ra\rc^{d}$
and $F\in\gsf(\rc^{d},\rc^{n})$, then we can define the composition
$F\circ(f_{1},\ldots,f_{n}):=(F\circ f_{1},\ldots,F\circ f_{n}):\bigcup_{i=1}^{n}[a_{i},a_{i+1}]\ra\rc^{n}$. 
\end{enumerate}
\end{rem}
Piecewise GSF inherit from their defining components a well-behaved
differential and integral calculus. The former is even more general
and taken from \cite{Ar-Fe-Ju05}.
\begin{defn}
Let $x=[x_{\eps}]\in\rc$, then

\begin{enumerate}
\item $\nu(x):=\sup\left\{ b\in\R\mid|x_{\eps}|=O(\rho_{\eps}^{b})\right\} \in\R\cup\{+\infty\}$.
\item $|x|_{\text{e}}:=e^{-\nu(x)}\in\R_{\ge0}$.
\item $\dd\rho(x):=\dd\rho^{-\log|x|_{\text{e}}}\in\rc_{>0}$.
\end{enumerate}
\end{defn}
\noindent It is worth noting that $|-|_{\text{e}}:\rc\ra\R_{\ge0}$
induces an ultrametric on $\rc$ that generates exactly the sharp
topology, see e.g.~\cite{AJ:01,GiKu13} and references therein. However,
we will not use this ultrametric structure in the present paper, and
we only introduced it to get an invertible infinitesimal $\dd\rho(x)$
that goes to zero with $x$: it is in fact easy to show that
\[
\lim_{x\to0}\frac{x}{\dd\rho(x)}=1
\]
in the sharp topology.
\begin{defn}
Let $T\subseteq\rc$ and let $f:T\ra\rc^{d}$ be an arbitrary set-theoretical
function. Let $t_{0}\in T$ be a sharply interior point of $T$. Then
we say that $f$ \emph{is differentiable at $t_{0}$} if\footnote{This definition is based on \cite[Def. 2.2]{Ar-Fe-Ju05}.}
\begin{equation}
\exists m\in\rc^{d}:\ \lim_{h\to0}\frac{f(t+h)-f(t_{0})-m\cdot h}{\dd\rho(h)}=0.\label{eq:differentiability}
\end{equation}
\end{defn}
\noindent In this case, using Landau little-oh notation, we can hence
write
\begin{equation}
f(t+h)=f(t_{0})+m\cdot h+o(\dd\rho(h))\quad\text{as }h\to0.\label{eq:Peano1st}
\end{equation}
As in the classical case, \eqref{eq:Peano1st} implies the uniqueness
of $m\in\rc^{d}$, so that we can define $f'(t_{0}):=\dot{f}(t_{0}):=m$,
and the usual elementary rules of differential calculus. By the Fermat-Reyes
theorem, this definition of derivative generalizes that given for
GSF.

In particular, this notion of derivative applies to the set-theoretical
function induced by a piecewise GSF $(f_{1},\ldots,f_{n})$. We therefore
have that $(f_{1},\ldots,f_{n})(-)$ is differentiable at each $a_{i}<t<a_{i+1}$,
and $(f_{1},\ldots,f_{n})'(t)=f_{i}'(t)$, but clearly there is no guarantee
that $(f_{1},\ldots,f_{n})(-)$ is also differentiable at each point
$a_{i}$.

The notion of definite integral is naturally introduced in the following
\begin{defn}
\label{def:intPGSF}Let $(f_{1},\ldots,f_{n}):\bigcup_{i=1}^{n}[a_{i},a_{i+1}]\ra\rc^{d}$
be a piecewise GSF, then
\[
\int_{a_{1}}^{a_{n+1}}(f_{1},\ldots,f_{n})(t)\diff{t}:=\sum_{i=1}^{n}\int_{a_{i}}^{a_{i+1}}f_{i}(t)\diff{t}.
\]
\end{defn}
\noindent Since our main aim in using piecewise GSF is to prove Jacobi's
theorem, we do not need to prove that the usual elementary rules of
integration hold, since we will always reduce to integrals of GSF.

Having a notion of derivative and of definite integral, also for piecewise
GSF we can consider functionals
\begin{equation}
\nu:=(f_{1},\ldots,f_{n}),\ a_{1}=a,\ a_{n}=b\then I(\nu):=\int_{a}^{b}F(t,\nu(t),\dot{\nu}(t))\dd t\in\rc,\label{eq:functionalsPGSF}
\end{equation}
and the concept of \emph{piecewise GSF (global) minimizer}: $I(\nu)\le I(\tilde{\nu})$
for all $\tilde{\nu}\in\gsf_{0}$. For the proof of Jacobi's theorem,
we will only need this particular notion of global minimizer. Note
explicitly that in \eqref{eq:functionalsPGSF} we only need existence
of right and left derivatives of GSF, because of Def.~\ref{def:intPGSF}
and of Def.~\ref{def:integral} of definite integral of GSF.

Classically, several proofs of Jacobi's theorem use both some form
of implicit function theorem and of uniqueness of solution for linear
ODE.
\begin{thm}[Implicit function theorem]
\label{thm:ImFT}Let $U\subseteq\rc^{n}$, $V\subseteq\rc^{d}$ be
sharply open sets. Let $F\in\gsf(U\times V,\rc^{d})$ and $(x_{0},y_{0})\in U\times V$.
If $\partial_{2}F(x_{0},y_{0})$ is invertible in $L(\rti^{d},\rti^{d})$,
then there exists a sharply open neighbourhood $U_{1}\times V_{1}\subseteq U\times V$
of $(x_{0},y_{0})$ such that
\begin{equation}
\forall x\in U_{1}\,\exists!y_{x}\in V_{1}:\ F(x,y_{x})=F(x_{0},y_{0}).\label{eq:implicitEq}
\end{equation}
Moreover, the function $f(x):=y_{x}$ for all $x\in U_{1}$ is a GSF
$f\in\gsf(U_{1},V_{1})$ and satisfies
\begin{equation}
Df(x)=-\left(\partial_{2}F(x,f(x))\right)^{-1}\circ\partial_{1}F(x,f(x)).\label{eq:implicitDer}
\end{equation}
\end{thm}
\begin{proof}
The usual deduction of the implicit function theorem from the inverse
function theorem in Banach spaces can be easily adapted using Thm.~\ref{thm:localIFTSharp}
and noting that $\det\left[\partial_{2}F(-,-)\right]$ is a GSF
such that $\left|\det\left[\partial_{2}F(x_{0},y_{0})\right]\right|\in\rc_{>0}$.
\end{proof}
In the next theorem, the dependence of the entire theory on the
initial infinitesimal net $\rho=(\rho_{\eps})\downarrow0$ plays an
essential role. Indirectly, the same important role will reverberate
in the final Jacobi's theorem.
\begin{thm}[Solution of first order linear ODE]
\label{thm:linearODE}Let $A\in\gsf([a,b],\rc^{d\times d})$, where
$a$, $b\in\rc$, $a<b$, and $t_{0}\in[a,b]$, $y_{0}\in\rc^{d}$.
Assume that
\begin{equation}
\left|\int_{t_{0}}^{t}A(s)\diff{s}\right|\le-C\cdot\log\dd\rho\quad\forall t\in[a,b],\label{eq:logHP}
\end{equation}
where $C\in\R_{>0}$. Then there exists one and only one $y\in\gsf([a,b],\rc^{d})$
such that
\begin{equation}
\begin{cases}
y'(t)=A(t)\cdot y(t) & \text{if }t\in[a,b]\\
y(t_{0})=y_{0}
\end{cases}\label{eq:LinearODE}
\end{equation}
Moreover, this $y$ is given by $y(t)=\exp\left(\int_{t_{0}}^{t}A(s)\diff{s}\right)\cdot y_{0}$
for all $t\in[a,b]$.
\end{thm}
\begin{proof}
We first note that $\exp\left(\int_{t_{0}}^{t}A(s)\diff{s}\right)=\left[\exp\left(\int_{t_{0\eps}}^{t_{\eps}}A_{\eps}(s)\diff{s}\right)\right]$,
where $t=[t_{\eps}]$, $t_{0}=[t_{0\eps}]$ and $A(s)=[A_{\eps}(s_{\eps})]\in\rc^{d\times d}$.
This exponential matrix in $\rc^{d\times d}$ is a GSF because for
all $t\in[a,b]$, we have
\[
\exp\left(\int_{t_{0}}^{t}A(s)\diff{s}\right)\le e^{-C\log\dd\rho}\le\dd\rho^{-C}.
\]
Therefore, all values of $y(t)=\exp\left(\int_{t_{0}}^{t}A(s)\diff{s}\right)\cdot y_{0}$
are $\rho$-moderate. Analogously, one can prove that also $y^{(k)}(t)$
are moderate for all $k\in\N$ and $t\in[a,b]$. Considering that
derivatives can be calculated $\eps$-wise, we have that this GSF
$y$ satisfies \eqref{eq:LinearODE}, and this proves the existence
part.

To show uniqueness, we can proceed as in the smooth case. Assume that
$z\in\gsf([a,b],\rc^{d})$ satisfies \eqref{eq:LinearODE}, and set
$h(t):=\exp\left(-\int_{t_{0}}^{t}A(s)\diff{s}\right)$ for all $t\in[a,b]$.
Since $h'=-A\cdot h$, we have $(hz)'=h'z+hz'=-Ahz+hAz = -Ahz + Ahz=0$. 
From uniqueness of primitives of GSF, Thm.~\ref{thm:existenceUniquenessPrimitives},
we have that $h\cdot z=h(t_{0})\cdot z(t_{0})=y_{0}$. Therefore $z=h^{-1}\cdot y_{0}$.
\end{proof}
If $\alpha$, $\beta\in\rc$, we write $\alpha=O_{\R}(\beta)$ to
denote that there exists $C\in\R_{>0}$ such that $|\alpha|\le C\cdot|\beta|$.
Therefore, assumption \eqref{eq:logHP} can be written as $\int_{t_{0}}^{t}A(s)\diff{s}=O_{\R}(\log\dd\rho)$.
Note that this assumption is weaker, in general, than
\[
(b-a)\cdot\max_{t\in[a,b]}|A(t)|=O_{\R}(\log\dd\rho).
\]

The following result is the key regularity property we need to prove
Jacobi's theorem.
\begin{lem}
\label{lem:regularityResult}Let $a$, $a'$, $b\in\rc$, with $a<a'<b$,
and let $K\in\gsf([a,b]\times\rc^{d}\times\rc^{d},\rc)$. Let $\nu=(\eta,\beta):[a,a']\cup[a',b]\ra\rc^{d}$
be a piecewise GSF which satisfies the Euler-Lagrange equation
\begin{equation}
K_{u}(t,\nu(t),\dot{\nu}(t))-\der{t}K_{\dot{u}}(t,\nu(t),\dot{\nu}(t))=0\quad\forall t\in[a,a')\cup(a',b].\label{eq:ELEregularityLem}
\end{equation}
Finally, assume that $\det\left(K_{\dot{u}_{i}\dot{u}_{j}}(a',\eta(a'),\dot{\eta}(a'))_{i,j=i,\ldots,d}\right)\in\rc$
is invertible, then
\begin{equation}
\lim_{\substack{t\to a'\\
t<a'
}
}\dot{\nu}(t)=\lim_{\substack{t\to a'\\
a'<t
}
}\dot{\nu}(t)=\dot{\eta}(a').\label{eq:limitsDxSin}
\end{equation}
In particular, if $\beta\equiv0|_{[a',b]}$, then $\dot{\eta}(a')=0$.
\end{lem}
\begin{proof}
Set $\Phi(t,l,v,q):=K_{\dot{u}}(t,l,v)-q$ for all $t\in[a,b]$ and
all $l$, $v$, $q\in\rc^{d}$. For simplicity, set $(t_{0},l_{0},v_{0},q_{0}):=(a',\eta(a'),\dot{\eta}(a'),K_{\dot{u}}(a',\eta(a'),\dot{\eta}(a'))$.
Our assumption on the invertibility of $K_{\dot{u}\dot{u}}(a',\eta(a'),\dot{\eta}(a'))=\partial_{v}\Phi(t_{0},l_{0},v_{0},q_{0})$
makes it possible to apply the implicit function Thm.~\ref{thm:ImFT}
to conclude that there exists a neighbourhood $T\times L\times V\times Q$
of $(t_{0},l_{0},v_{0},q_{0})$ such that
\begin{equation}
\forall(t,l,q)\in T\times L\times Q\,\exists!v\in V:\ \Phi(t,l,v,q)=\Phi(t_{0},l_{0},v_{0},q_{0}).\label{eq:uniquenessImFT}
\end{equation}
But $\Phi(t_{0},l_{0},v_{0},q_{0})=K_{\dot{u}}(a',\eta(a'),\dot{\eta}(a'))-q_{0}=0$.
Moreover, the unique function $\phi$ defined by $\Phi(t,l,\phi(t,l,q),q)=0$
for all $(t,l,q)\in T\times L\times Q$ is a GSF $\phi\in\gsf(T\times L\times Q,V)$.
Now, for all $t\in[a,a')\cup(a',b]$, we have
\[
\Phi(t,\nu(t),\dot{\nu}(t),K_{\dot{u}}(t,\nu(t),\dot{\nu}(t)))=K_{\dot{u}}(t,\nu(t),\dot{\nu}(t))-K_{\dot{u}}(t,\nu(t),\dot{\nu}(t))=0.
\]
Therefore, uniqueness in \eqref{eq:uniquenessImFT} yields
\begin{equation}
\dot{\nu}(t)=\phi\left(t,\nu(t),K_{\dot{u}}(t,\nu(t),\dot{\nu}(t))\right)\quad\forall t\in[a,a')\cup(a',b].\label{eq:nuDotImplicit}
\end{equation}
We now integrate the Euler-Lagrange equation \eqref{eq:ELEregularityLem}
on $[a,t]$, obtaining
\[
K_{\dot{u}}(t,\nu(t),\dot{\nu}(t))=\int_{a}^{t}K_{u}(s,\nu(s),\dot{\nu}(s))\diff{s}+K_{\dot{u}}(a,\eta(a),\dot{\eta}(a))\quad\forall t\in[a,a')\cup(a',b].
\]
This entails that we can write
\begin{equation}
\dot{\nu}(t)=\phi\left(t,\nu(t),\int_{a}^{t}K_{u}(s,\nu(s),\dot{\nu}(s))\diff{s}+K_{\dot{u}}(a,\eta(a),\dot{\eta}(a))\right)\quad\forall t\in[a,a')\cup(a',b].\label{eq:nuDotInt}
\end{equation}
But the function $t\in[a,a')\cup(a',b]\mapsto\int_{a}^{t}K_{u}(s,\nu(s),\dot{\nu}(s))\diff{s}\in\rc^{d}$
has equal limits on the left and on the right of $a'$ because on
$[a,a')$ and on $(a',b]$ it is a GSF; in fact for $t<a'$ we have
\begin{multline*}
\left|\int_{a}^{t}K_{u}(s,\nu(s),\dot{\nu}(s))\diff{s}-\int_{a}^{a'}K_{u}(s,\nu(s),\dot{\nu}(s))\diff{s}\right|\le\\
\le\max_{t\in[a,a']}\left|K_{u}(s,\eta(s),\dot{\eta}(s))\right|\cdot|t-a'|,
\end{multline*}
and this goes to $0$ as $t\to a'$, $t<a'$. Analogously we can proceed for $t>a'$ using $\beta$.
Therefore
\[
\lim_{\substack{t\to a'\\
t<a'
}
}\int_{a}^{t}K_{u}(s,\nu(s),\dot{\nu}(s))\diff{s}=\lim_{\substack{t\to a'\\
t>a'
}
}\int_{a}^{t}K_{u}(s,\nu(s),\dot{\nu}(s))\diff{s}.
\]
Applying this equality in \eqref{eq:nuDotInt}, we get $\lim_{\substack{t\to a'\\
t<a'
}
}\dot{\nu}(t)=\dot{\eta}(a')=\lim_{\substack{t\to a'\\
a'<t
}
}\dot{\nu}(t)$ as claimed. Finally, if $\beta\equiv0|_{[a',b]}$, then $\lim_{\substack{t\to a'\\
a'<t
}
}\dot{\nu}(t)=0$.
\end{proof}
\noindent In the following definition and below, we use the complete
notation $\gsf_{0}(a,a')$ (see Def.~\ref{def:gsf_0}).
\begin{defn}
\label{def:conjPt}Let $a$, $a'$, $b\in\rc$, where $a<a'<b$. We
call $a'$ \emph{conjugate to} $a$ w.\,r.\,t.\ the variational problem \eqref{eq:def_I} if there exists a non identically
vanishing Jacobi field $\eta\in\gsf_{0}(a,a')$ along $u|_{[a,a']}$
such that $\eta(a)=0=\eta(a')$, where $\psi$ is given by \eqref{eq:aux_prob_psi}.
\end{defn}
In order to prove the important Jacobi's theorem in the present generalized
context, which shows that we cannot have minimizers if there are interior
points which are conjugate to $a$, we finally need the following 
\begin{lem}
\label{lem:jacobi_zero}Let $u\in\gsf([a,b],\rc^{d})$ and let $a'\in(a,b)$.
Let $\eta\in\gsf_{0}(a,a')$ be a Jacobi field along $u|_{[a,a']}$,
with $\eta(a)=0=\eta(a')$. Then
\[
\int_{a}^{a'}\psi(t,\eta,\dot{\eta})\,\dd t=0.
\]
\end{lem}
\begin{proof}
Since $\psi$ is $\rc$-homogeneous of second order in $(\eta,\dot{\eta})$,we
have
\[
2\psi(t,\eta,\dot{\eta})=\psi_{\eta}(t,\eta,\dot{\eta})\eta+\psi_{\dot{\eta}}(t,\eta,\dot{\eta})\dot{\eta}.
\]
Thus we calculate: 
\begin{align*}
2\int_{a}^{a'}\psi(t,\eta,\dot{\eta})\,\dd t & =\int_{a}^{a'}\eta\psi_{\eta}(t,\eta,\dot{\eta})+\dot{\eta}\psi_{\dot{\eta}}(t,\eta,\dot{\eta})\,\dd t\\
 & =\int_{a}^{a'}\eta\left(\psi_{\eta}(t,\eta,\dot{\eta})-\der{t}\psi_{\dot{\eta}}(t,\eta,\dot{\eta})\right)\,\dd t\:\:\text{by integration by parts}\\
 & =0\text{ since }\eta\text{ is a Jacobi field}.
\end{align*}
\end{proof}
After these preparations we can finally prove
\begin{thm}[Jacobi]
\label{thm:jacobi_no_min}Let $a$, $b\in\rc$ be such that $a<b$.
Suppose that $F\in\gsf([a,b]\times\rc^{d}\times\rc^{d},\rc)$ and
$u\in\gsf([a,b],\rc)$ are such that

\begin{enumerate}
\item $a'\in(a,b)$ is conjugate to $a$
\item \label{enu:invertibilityAssJacobi}$\det F_{\dot{u}\dot{u}}(t,u(t),\dot{u}(t))\in\rc$
is invertible for all $t\in[a,b]$.
\item \label{enu:log}For all $t\in[a,a']$
\begin{multline*}
\int_{a'}^{t}-F_{\dot{u}\dot{u}}^{-1}(s,u(s),\dot{u}(s))\cdot\left[\der{s}F_{u\dot{u}}(s,u(s),\dot{u}(s))-F_{uu}(s,u(s),\dot{u}(s))\right]\diff{s}=\\
=O_{\R}(\log\dd\rho)
\end{multline*}
\[
\int_{a'}^{t}-F_{\dot{u}\dot{u}}^{-1}(s,u(s),\dot{u}(s))\cdot\der{s}F_{\dot{u}\dot{u}}(s,u(s),\dot{u}(s))\diff{s}=O_{\R}(\log\dd\rho).
\]
\end{enumerate}
Then $u$ cannot be a local minimizer of $I$. Therefore, for any
$r\in\rc_{>0}$ there exists $v\in\gsf_{\text{bd}}(u(a),u(b))$ and
$m\in\N$ such that $\Vert v-u\Vert_{m}<r$ but $I(u)\not\le I(v)$.

\end{thm}
\begin{proof}
By contradiction, assume that $u$ is a local minimizer, and let $\eta\in\gsf_{0}(a,a')$
be a Jacobi field along $u|_{[a,a']}$ such that the conditions from Def.~\ref{def:conjPt}
hold for $\eta$. We want to prove that $\eta\equiv0$. Define $\nu:=(\eta,0|_{[a',b]})$,
which is a piecewise GSF since $\eta(a')=0$. Since also $\eta(a)=0$,
Lem.~\ref{lem:jacobi_zero} and homogeneity of $\psi$ yield
\[
Q(\nu)=\int_{a}^{b}\psi(t,\nu(t),\dot{\nu}(t))\diff{t}=\int_{a}^{a'}\psi(t,\eta(t),\dot{\eta}(t))\diff{t}+\int_{a'}^{b}\psi(t,0,0)\diff{t}=0.
\]
Therefore, Thm.~\ref{thm:necessCondsForMinimizer} (necessary condition
for $u$ being a minimizer) gives $Q(\tilde{\nu})\ge0=Q(\nu)$ for
all $\tilde{\nu}\in\gsf_{0}(a,b)$. Thus, $\nu$ is a minimizer of
the functional $Q$. Since $\nu$ is only a piecewise GSF, we cannot
directly apply Thm.~\ref{thm:ELE} (Euler-Lagrange equations). But
for all $\phi\in\gsf_{0}(a,b)$ and all $s\in\rc$, we have
\begin{align}
Q(\nu+s\phi) & =\int_{a}^{b}\psi(t,\nu+s\phi,\dot{\nu}+s\dot{\phi})\diff{t}\nonumber \\
 & =\int_{a}^{a'}\psi(t,\eta+s\phi,\dot{\eta}+s\dot{\phi})\diff{t}+\int_{a'}^{b}\psi(t,s\phi,s\dot{\phi})\diff{t}.\label{eq:Q}
\end{align}
This shows that $s\in\rc\mapsto Q(\nu+s\phi)\in\rc$ is a GSF, and
hence $s=0$ is a minimum for this function. By Lem.~\ref{lem:local_min_diff}
and \eqref{eq:Q}, we get
\begin{align*}
\delta Q(\nu,\phi) & =0=\left.\der{s}Q(\nu+s\phi)\right|_{0}\\
 & =\int_{a}^{a'}\left(\psi_{\eta}(t,\eta,\dot{\eta})-\der{t}\psi_{\dot{\eta}}(t,\eta,\dot{\eta})\right)\phi\diff{t}+\int_{a'}^{b}\left(\phi\psi_{\eta}(t,0,0)+\dot{\phi}\psi_{\dot{\eta}}(t,0,0)\right)\diff{t}\\
 & =\int_{a}^{a'}\left(\psi_{\eta}(t,\eta,\dot{\eta})-\der{t}\psi_{\dot{\eta}}(t,\eta,\dot{\eta})\right)\phi\diff{t}.
\end{align*}
By the fundamental Lem.~\ref{lem:fund_lem_calc_var}, this implies
that $\eta$ satisfies the Euler-Lagrange equations for $\psi$ in
the interval $[a,a')$. Therefore, $\nu$ satisfies the same equations
in $[a,a')\cup(a',b]$. Moreover, $\psi_{\dot{\eta}\dot{\eta}}(a',\eta(a'),\dot{\eta}(a'))=F_{\dot{u}\dot{u}}(a',u(a'),\dot{u}(a'))$
is invertible by assumption \ref{enu:invertibilityAssJacobi}. Thus,
all the hypotheses of the regularity Lem\@.~\ref{lem:regularityResult}
hold, and we derive that $\dot{\eta}(a')=0$.

For all $t\in[a,b]$, we define 
\begin{align*}
\xi(t) & :=-F_{\dot{u}\dot{u}}^{-1}\cdot\left[\der{t}F_{u\dot{u}}(t,u,\dot{u})-F_{uu}(t,u,\dot{u})\right],\text{ and}\\
\vartheta(t) & :=-F_{\dot{u}\dot{u}}^{-1}\cdot\der{t}F_{\dot{u}\dot{u}}(t,u,\dot{u}),
\end{align*}
so that we can re-write the Jacobi equations \eqref{eq:jacobi} for
$\eta$ on $[a,a']$ as a system of first order ODE
\[
\begin{cases}
\dot{y}:=\begin{pmatrix}\dot{\eta}\\
\dot{z}
\end{pmatrix}=\begin{pmatrix}0 & 1\\
\xi & \vartheta
\end{pmatrix}\cdot\begin{pmatrix}\eta\\
z
\end{pmatrix}=:A\cdot y & \quad\forall t\in[a,a']\\
y(a')=\begin{pmatrix}\eta(a')\\
\dot{\eta}(a')
\end{pmatrix}=0
\end{cases}
\]
By assumptions \ref{enu:log}, we obtain $\int_{a'}^{t}A(t)=O_{\R}(\log\dd\rho)$
for all $t\in[a,a']$, and we can hence apply Thm.~\ref{thm:linearODE}
obtaining $y\equiv0$ and thus $\eta\equiv0$.
\end{proof}
\noindent Note that if one of the quantities in \eqref{enu:log} depends
even only polynomially on $\eps$, then we are forced to take e.g.~$\rho_{\eps}=\eps^{1/\eps}$
to fulfill this assumption. This underlines the importance of
the parameter $\rho$ making the entire theory dependent on the parameter $\rho$., in order to avoid
unnecessary constraints on the scope of the functionals we look upon.

\section{\label{sec:Noether's-theorem}Noether's theorem}

In this section, we state and prove Noether's theorem following the
lines of \cite{Ave86}. We first note that any $X\in\gsf\left(J\times X,Y\right)$,
where $J\subseteq\rc$, can also be considered as a family in GSF
which smoothly depends on the parameter $s\in J$. In this case, we
hence say that $(X_{s})_{s\in J}$ \emph{is a generalized smooth family
in} $\gsf(X,Y)$. In particular, we can reformulate in the language
of GSF the classical definition of \emph{one-parameter group of generalized
diffeomorphisms of $X$} as follows:
\begin{enumerate}
\item \label{enu:GSFam}$(X_{s})_{s\in{\rc}}$ is a generalized smooth family
in $\gsf(X,X)$,
\item For all $s\in\rc$, the map $X_{s}:X\ra X$ is invertible, and $X_{s}^{-1}\in\gsf(X,X)$,
\item \label{enu:Phi_0-Id}$X_{0}(x)=x$ for all $x\in X$,
\item $X_{s}\circ X_{t}=X_{s+t}$ for all $s$, $t\in\rc$.
\end{enumerate}
\noindent In our proofs, we will in fact only use properties \ref{enu:GSFam}
and \ref{enu:Phi_0-Id}.

The proof of Noether's theorem is classically anticipated by the following
time-independent version, which the general case is subsequently
reduced to.
\begin{thm}
\label{thm:aux_noether}Let $K\in\gsf\left(L\times V,\rc\right)$,
where $L$, $V\subseteq\rc^{n}$ are sharply open sets. Let $w\in\gsf((a,b),L)$
be a solution of the Euler-Lagrange equation corresponding to $K$,
i.e.~for all $t\in(a,b)$
\begin{equation}
\dot{w}(t)\in V\ ,\ K_{u}(w(t),\dot{w}(t))=\der{t}K_{\dot{u}}(w(t),\dot{w}(t)).\label{eq:ELE-K}
\end{equation}
Suppose that $0$ is a sharply interior point of $J\subseteq\rc$
and $(X_{s})_{s\in J}$ is a generalized smooth family in $\gsf(L,L)$,
such that for all $t\in(a,b)$

\begin{enumerate}
\item $\pdd{t}X_{s}(w(t))\in V$,
\item \label{enu:X_0-Id}$X_{0}(w(t))=w(t)$,
\item \label{enu:I-notDepOn-s}$K$ is invariant under $(X_{s})_{s\in J}$
along $w$, i.e.
\begin{equation}
K(w(t),\dot{w}(t))=K\left(X_{s}(w(t)),\pdd{t}X_{s}(w(t))\right)\quad\forall s\in J.\label{eq:invK}
\end{equation}
\end{enumerate}
Then, the quantity
\[
K_{\dot{u}^{j}}(w(t),\dot{w}(t))\left.\pdd{s}\right|_{s=0}X_{s}^{j}(w(t))\in\rc
\]
is constant in $t\in(a,b)$.
\end{thm}
\begin{proof}
We first note that both sides of \eqref{eq:invK} are in $\gsf((a,b),\rc)$.
Let $\tau\in(a,b)$ be arbitrary but fixed. Since $s=0\in J$ is a
sharply interior point, we can consider $\left.\frac{\dd}{\dd s}\right|_{s=0}$.
We obtain
\begin{align*}
0 & \stackrel{\eqref{eq:invK}}{=}\left.\pdd{s}\right|_{s=0}K\left(X_{s}(w),\pdd{t}X_{s}(w)\right)\\
 & \stackrel{\ref{enu:X_0-Id}}{=}\int_{a}^{\tau}K_{u}(w,\dot{w})\left.\pdd{s}\right|_{s=0}X_{s}(w)+K_{\dot{u}}(w,\dot{w})\pdd{t}\left.\pdd{s}\right|_{s=0}X_{s}(w)\,\dd t.
\end{align*}
Since the Euler-Lagrange equations \eqref{eq:ELE-K} for $K$ are
given by $K_{u}(w,\dot{w})=\der{t}K_{\dot{u}}(w,\dot{w})$, we have
\begin{align*}
0 & =\der{t}\left(K_{\dot{u}}(w,\dot{w})\right)\left.\pdd{s}\right|_{s=0}X_{s}(w)+K_{\dot{u}}(w,\dot{w})\pdd{t}\left.\pdd{s}\right|_{s=0}X_{s}(w)\\
 & =\der{t}\left(K_{\dot{u}}(w,\dot{w})\left.\pdd{s}\right|_{s=0}X_{s}(w)\right).
\end{align*}
Which is our conclusion by the uniqueness - part of Thm\@.~\ref{thm:existenceUniquenessPrimitives}.
\end{proof}
We are now able to prove Noether's theorem. For the convenience of the reader,
in its statement and proof we use the variables $t$, $T$, $l$,
$L$, $v$, $V$ so as to recall \emph{tempus}, \emph{locus}, \emph{velocitas}
resp.
\begin{thm}[E.~Noether]
\label{thm:noether}~\\
Let $a$, $b\in\rc^{d}$, with $a<b$, and $F\in\gsf\left([a,b]\times\rc^{d}\times\rc^{d},\rc\right)$.
Let $u\in\gsf([a,b],\rc^{d})$ be a solution of the Euler-Lagrange
equation \eqref{eq:ELE} corresponding to $F$. Suppose that $0$
is a sharply interior point of $J\subseteq\rc$ and $(X_{s})_{s\in J}$
is a generalized smooth family in $\gsf((a,b)\times\rc^{d},(a,b)\times\rc^{d})$.
We denote by $T_{s}(t,l):=X_{s}^{1}(t,l)\in(a,b)$ and $L_{s}(t,l):=X_{s}^{2}(t,l)\in\rc^{d}$
for all $(t,l)\in(a,b)\times\rc^{d}$, the two projections of $X_{s}$
on $(a,b)$ and $\rc^{d}$ resp. We assume that for all $t\in(a,b)$

\begin{enumerate}
\item \label{enu:pddT_sInv}$\pdd{t}T_{s}(t,u(t))\in\rc$ is invertible,
\item \label{enu:T_0-L_0}$T_{0}(t,u(t))=t$ and $L_{0}(t,u(t))=u(t)$,
\item \label{enu:FInv}$F(t,u(t),\dot{u}(t))=F\left[T_{s}(t,u),L_{s}(t,u),\frac{\pdd{t}L_{s}(t,u)}{\pdd{t}T_{s}(t,u)}\right]\cdot\pdd{t}T_{s}(t,u)$
for all $s\in J$.
\end{enumerate}
Then, the quantity
\begin{multline}
F_{\dot{u}^{j}}(t,u(t),\dot{u}(t))\left.\pdd{s}\right|_{s=0}L_{s}^{j}(t,u(t))+\\
+\left[F(t,u(t),\dot{u}(t))-F_{\dot{u}^{k}}(t,u(t),\dot{u}(t))\dot{u}^{k}(t)\right]\left.\pdd{s}\right|_{s=0}T_{s}(t,u(t))\label{eq:intMotNoether}
\end{multline}
is constant in $t\in[a,b]$.
\end{thm}
\begin{proof}
Since \eqref{eq:intMotNoether} is a GSF in $t\in[a,b]$, by sharp
continuity it suffices to prove the claim for all $t\in(a,b)$. Set
$L:=(a,b)\times\rc^{d}$, $V:=\rc^{*}\times\rc^{d}$ (we recall that
$\rc^{*}$ denotes the set of all invertible generalized numbers in
$\rc$). Define $K\in\gsf(L\times V,\rc)$ by
\begin{equation}
K(t,l;p,v):=F\left(t,l,\frac{v}{p}\right)\cdot p\quad\forall(t,l)\in L\,\forall(p,v)\in V,\label{eq:defK}
\end{equation}
and $w\in\gsf((a,b),L)$ by $w(t):=(t,u(t))$ for all $t\in(a,b)$.
We note that $L$, $V\subseteq\rc^{d+1}$ are sharply open subsets
and that $\dot{w}(t)=(1,\dot{u}(t))\in V$. The notations for partial
derivatives used in the present work result from the symbolic writing
$K(u^{1},\ldots,u^{d+1};\dot{u}^{1},\ldots,\dot{u}^{d+1})$, so that
the variables used in \eqref{eq:defK} yield
\begin{equation}
K_{u^{j}}(t,l;p,v)=\begin{cases}
K_{t}(t,l;p,v)=F_{t}\left(t,l,\frac{v}{p}\right)\cdot p & \text{ if }j=1\\
K_{l^{j}}(t,l;p,v)=F_{u^{j-1}}\left(t,l,\frac{v}{p}\right)\cdot p & \text{ if }j=2,\ldots,d+1,
\end{cases}\label{eq:K_u^j}
\end{equation}
and
\begin{equation}
K_{\dot{u}^{j}}(t,l;p,v)=\begin{cases}
K_{p}(t,l;p,v)=F\left(t,l,\frac{v}{p}\right)-F_{\dot{u}^{k}}\left(t,l,\frac{v}{p}\right)\frac{v^{k}}{p} & \text{ if }j=1\\
K_{v^{j}}(t,l;p,v)=F_{\dot{u}^{j-1}}\left(t,l,\frac{v}{p}\right) & \text{ if }j=2,\ldots,d+1.
\end{cases}\label{eq:K_udot^j}
\end{equation}
From these, for all $t\in(a,b)$ and all $j=2,\ldots,d+1$, it follows that
\begin{align*}
K_{u^{1}}(w,\dot{w})-\der{t}K_{\dot{u}^{1}}(w,\dot{w}) & =\left[\der{t}F_{\dot{u}^{k}}(t,u,\dot{u})-F_{u^{k}}(t,u,\dot{u})\right]\cdot\dot{u}^{k}\\
K_{u^{j}}(w,\dot{w})-\der{t}K_{\dot{u}^{j}}(w,\dot{w}) & =F_{u^{j-1}}(t,u,\dot{u})-\der{t}F_{\dot{u}^{j-1}}(t,u,\dot{u}).
\end{align*}
Therefore, since $u$ satisfies the Euler-Lagrange equations for $F$,
this entails that $w$ is a solution of the analogous equations for $K$
in $(a,b)$. Now, \ref{enu:pddT_sInv} gives
\[
\pdd{t}X_{s}(w(t))=\left(\pdd{t}T_{s}(t,u(t)),\pdd{t}L_{s}(t,u(t))\right)\in\rc^{*}\times\rc^{d}=V.
\]
Moreover, \ref{enu:T_0-L_0} gives $X_{0}(w(t))=\left(T_{0}(t,u(t)),L_{0}(t,u(t))\right)=w(t)$.
Finally
\begin{align*}
K(w,\dot{w}) & =F(t,u,\dot{u})\\
K\left(X_{s}(w),\pdd{t}X_{s}(w)\right) & =F\left[T_{s}(t,u),L_{s}(t,u),\frac{\pdd{t}L_{s}(t,u)}{\pdd{t}T_{s}(t,u)}\right]\cdot\pdd{t}T_{s}(t,u).
\end{align*}
We can hence apply Thm.~\ref{thm:aux_noether}, and from \eqref{eq:K_u^j},
\eqref{eq:K_udot^j} we get that
\begin{multline*}
K_{\dot{u}^{j}}(w,\dot{w})\left.\pdd{s}\right|_{0}X_{s}^{j}(w)=F_{\dot{u}^{j}}(t,u(t),\dot{u}(t))\left.\pdd{s}\right|_{s=0}L_{s}^{j}(t,u(t))+\\
+\left[F(t,u(t),\dot{u}(t))-F_{\dot{u}^{k}}(t,u(t),\dot{u}(t))\dot{u}^{k}(t)\right]\left.\pdd{s}\right|_{s=0}T_{s}(t,u(t))
\end{multline*}
is constant in $t\in(a,b)$.
\end{proof}

\section{\label{sec:Application}Application to $\mathcal{C}^{1,1}$ Riemannian
metric}

In the following, we apply what we did so far to the problem of length-minimizers
in $(\R^{d},g)$, where $g\in\mathcal{C}^{1,1}$ is a Riemannian metric.
Furthermore, we assume that $(\R^{d},g)$ is geodesically complete.
Note that the seeming restriction of considering only $\R^{d}$
as our manifold weighs not so heavy. Indeed, the question of length
minimizers can be considered to be a local one, since it is not guaranteed
that global minimizers exist at all, whereas local minimizers always
exist. Additionally, note that it was shown that it suffices to consider
smooth manifolds (cf.\ \cite[Thm. 2.9]{Hi:76}) instead of $\mathcal{C}^{k}$
manifolds with $1\le k<+\infty$. Therefore, there is no need to consider
non-smooth charts.

In this section, we fix any embedding $(\iota_{\Omega}^{b})_{\Omega}$,
where $b\in\rc$ satisfies $b\ge\dd\rho^{-a}$ for some $a\in\R_{>0}$,
and where $\Omega\subseteq\R^{d}$ is an arbitrary open set, see Thm.~\ref{thm:embeddingD'}.
Actually, the embedding also depends on the dimension $d\in\N_{>0}$,
but to avoid cumbersome notations, we denote embeddings always with
the symbol $\iota$.

By \cite[Rem. 2.6.2]{KSSV:13}, it follows that we can always find a
net of smooth functions $(g_{ij}^{\eps})$ such that setting $\tilde{g}:=\iota(g)=:[g_{ij}^{\eps}(-)]\in\gsf(\rc^{d}\times\rc^{d},\rc)$,
then for all $\eps$, $g_{ij}^{\eps}$ is a Riemannian metric. By Thm.~\ref{thm:embeddingD'}.\ref{enu:eps-D'}
it follows that $g_{ij}^{\eps}\to g_{ij}$ in $\mathcal{C}^{0}$ norm.
Let $\Gamma_{ij}^{\eps}$ be the Christoffel symbols of $g^{\eps}$,
and set $\tilde{\Gamma}_{ij}:=[\Gamma_{ij}^{\eps}(-)]\in\gsf(\rc^{d},\rc^{d})$.
A curve $\gamma\in\gsf(J,\rc^{d})$, $J$ being a sharply open subset
of $\rc$, is said to be a \emph{geodesic} of $(\rc^{d},\tilde{g})$
if
\begin{equation}
\ddot{\gamma}(t)+\tilde{\Gamma}_{ij}(\gamma(t))\dot{\gamma}^{i}(t)\dot{\gamma}^{j}(t)=0\quad\forall t\in J.\label{eq:geoEq}
\end{equation}

\begin{defn}
We say that $(\rc^{d},\tilde{g})$ is \emph{geodesically complete}
if every solution of the geodesic equation belongs to $\gsf(\rc,\rc^{d})$,
i.e.~if for all $p\in\rc^{d}$ and all $v\in\rc^{d}$ there exists
a geodesic $\gamma\in\gsf(\rc,\rc^{d})$ of $(\rc^{d},\tilde{g})$
such that $\gamma(0)=p$ and $\dot{\gamma}(0)=v$.
\end{defn}
\noindent This definition includes also the possibility that the point
$p$ or the vector $v$ could be infinite. By Thm.~\ref{thm:inclusionCGF},
it follows that if we consider only finite $p$ and $v$, then any
geodecis $\gamma\in\gsf(\rc,\rc^{d})$ induces a Colombeau generalized
function $\gamma|_{\csp{\R}}\in\gs(\R)^{d}$. Therefore, the space
$(\csp{\R^{d}},\tilde{g}|_{\csp{\R^{d}}\times\csp{\R^{d}}})$ is geodesically
complete in the sense of \cite{Sa-St15}. We recall that $\csp{\Omega}$
is the set of compactly supported (i.e.~finite) generalized points
in $\Omega$ (see Thm.~\ref{thm:embeddingD'}).

The definition of length of a (non singular) curve needs the following
\begin{rem}
We set
\[
\sqrt{-}=(-)^{1/2}:x=[x_{\eps}]\in\rc_{>0}\mapsto[\sqrt{x_{\eps}}]\in\rc_{>0}.
\]
Lem.~\ref{lem:mayer} readily implies that $\sqrt{-}\in\gsf(\rc_{>0},\rc_{>0})$.
Therefore, the square root is defined on every strictly positive infinitesimal,
but it cannot be extended to $\rc_{\ge0}$.
\end{rem}
\begin{defn}
~
\begin{enumerate}
\item Let $\tilde{p}$, $\tilde{q}\in\rc^{d}$, then
\[
\gsf_{>0}(\tilde{p},\tilde{q}):=\left\{ \lambda\in\gsf([0,1],\rc^{d})\mid\lambda(0)=\tilde{p},\lambda(1)=\tilde{q},|\dot{\lambda}(t)|>0\ \forall t\in[0,1]\right\} .
\]
Moreover, for $\lambda\in\gsf_{>0}(\tilde{p},\tilde{q})$, we set
\[
L_{\tilde{g}}(\lambda):=\int_{0}^{1}\left(\tilde{g}_{ij}(\alpha(t))\dot{\alpha}^{i}(t)\dot{\alpha}^{j}(t)\right)^{1/2}\diff{t}\in\rc.
\]
\item Let $x=[x_{\eps}]\in\rc^{n}$, then we set $\text{st}(x):=\lim_{\eps\to0}x_{\eps}\in\R^{d}$,
if this limit exists. Note that $x\approx\text{st}(x)$ in this case.
\end{enumerate}
\end{defn}
Note that \eqref{eq:geoEq} are the usual geodesic equations for the generalized
metric $\tilde{g}$, whose derivation is completely analogous
to that in the smooth case. Thus they are the Euler-Lagrange equations of $L_{\tilde{g}}$.

We are interested only in global minimizers of the functional $L_{\tilde{g}}$,
i.e.~curves $\lambda_{0}\in X(\tilde{p},\tilde{q})$ such that $L_{\tilde{g}}(\lambda_{0})\leq L_{\tilde{g}}(\lambda)$
for all $\lambda\in\gsf_{>0}(\tilde{p},\tilde{q})$. 
\begin{lem}
\label{lem:conv_length}Let $p,q\in\R^d$ and $\tilde{p}$, $\tilde{q}\in\rc^{d}$ such that $ \text{st}(\tilde p)=p$ and $ \text{st}(\tilde q)=q$. Let
$\lambda=[\lambda_{\varepsilon}(-)]\in\gsf_{>0}(\tilde{p},\tilde{q})$
be such that there exists
\[
\bar{\lambda}\in\mathcal{C}_{>0}^{1}(p,q):=\left\{ w\in\mathcal{C}^{1}([0,1]_{\R},\R^{d})\mid w(0)=p,w(1)=q,|\dot{w}(t)|>0\ \forall t\in[0,1]_{\R}\right\} 
\]
such that $\lambda_{\varepsilon}\to\bar{\lambda}$ in $\mathcal{C}^{1}$
as $\varepsilon\to0$, then\textup{
\[
\text{st}(L_{\tilde{g}}(\lambda))=L_{g}(\bar{\lambda}).
\]
}
\end{lem}
\begin{proof}
We calculate:
\begin{multline*}
\left|\int_{0}^{1}\left(g_{ij}^{\eps}(\lambda_{\varepsilon})\dot{\lambda}_{\varepsilon}^{i}\dot{\lambda}_{\varepsilon}^{j}\right)^{1/2}-\left(g_{ij}(\bar{\lambda})\dot{\bar{\lambda}}^{i}\dot{\bar{\lambda}}^{j}\right)^{1/2}\diff{t}\right|=\\
=\left|\int_{0}^{1}\frac{g_{ij}^{\eps}(\lambda_{\varepsilon})\dot{\lambda}_{\varepsilon}^{i}\dot{\lambda}_{\varepsilon}^{j}-g_{ij}(\bar{\lambda})\dot{\bar{\lambda}}^{i}\dot{\bar{\lambda}}^{j}}{(g_{ij}^{\eps}(\lambda_{\varepsilon})\dot{\lambda}_{\varepsilon}^{i}\dot{\lambda}_{\varepsilon}^{j})^{1/2}+(g_{ij}(\bar{\lambda})\dot{\bar{\lambda}}^{i}\dot{\bar{\lambda}}^{j})^{1/2}}\diff{t}\right|.
\end{multline*}
By assumption, $(g_{ij}^{\eps}(\lambda_{\varepsilon})\dot{\lambda}_{\varepsilon}^{i}\dot{\lambda}_{\varepsilon}^{j})^{1/2}\to(g_{ij}(\bar{\lambda})\dot{\bar{\lambda}}^{i}\dot{\bar{\lambda}}^{j})^{1/2}$,
so that there exists $C\in\R_{>0}$ such that
\begin{align*}
&\left|\int_{0}^{1}\left(g_{ij}^{\eps}(\lambda_{\varepsilon})\dot{\lambda}_{\varepsilon}^{i}\dot{\lambda}_{\varepsilon}^{j}\right)^{1/2}-\left(g_{ij}(\bar{\lambda})\dot{\bar{\lambda}}^{i}\dot{\bar{\lambda}}^{j}\right)^{1/2}\diff{t}\right|\\
&\quad\le C\int_{0}^{1}\left|(g_{ij}^{\eps}(\lambda_{\varepsilon})-g_{ij}(\lambda_{\varepsilon})+g_{ij}(\lambda_{\varepsilon})-g_{ij}(\bar{\lambda}))\dot{\lambda}_{\varepsilon}^{i}\dot{\lambda}_{\varepsilon}^{j}+g_{ij}(\bar{\lambda})(\dot{\lambda}_{\varepsilon}^{i}\dot{\lambda}_{\varepsilon}^{j}-\dot{\bar{\lambda}}^{i}\dot{\bar{\lambda}}^{j})\right|\diff{t}.
\end{align*}
We hence obtain the claim by the triangle inequality and by convergence
of $\lambda_{\varepsilon}$, $\dot{\lambda}_{\eps}$ and $g_{ij}^{\eps}$
to $\bar{\lambda}$, $\dot{\bar{\lambda}}$ and $g_{ij}$ respectively. 
\end{proof}
Now, we consider $p$, $q\in\R^{d}$ with $p\ne q$. Let
\[
u\in\left\{ u\in\mathcal{C}^{2,1}([0,1],\R^{d})\mid u(0)=p,u(1)=q\right\} 
\]
be a solution of the geodesic equation
\begin{equation}
\begin{cases}
\ddot{u}=-\Gamma_{ij}(u)\dot{u}^{i}\dot{u}^{j}\\
p=u(0)\\
q=u(1).
\end{cases}\label{eq:geo_eq_classical}
\end{equation}
Let $c_{0}:=\dot{u}(0)$. Obviously, $u$ is also the unique solution
of
\begin{equation}
\begin{cases}
\ddot{u}=-\Gamma_{ij}(u)\dot{u}^{i}\dot{u}^{j}\\
p=u(0)\\
c_{0}=\dot{u}(0).
\end{cases}\label{eq:geo_eq_classical_two}
\end{equation}
Using these initial conditions, for each fixed $\eps$ we can solve
the following problem
\begin{equation}
\begin{cases}
\ddot{y}=-\Gamma_{ij}^{\eps}(y)\dot{y}^{i}\dot{y}^{j}\\
p=y(0)\\
c_{0}=\dot{y}(0).
\end{cases}\label{eq:geo_eq_gen}
\end{equation}
for a unique $y_{\eps}\in\Coo([-d_{\eps},d_{\eps}]_{\R},\R^{d})$
and some $d_{\eps}\in\R_{>0}$.
\begin{lem}
\label{lem:y_is_gsf}Let $u$ and $y_{\varepsilon}$ be as above.
Then

\begin{enumerate}
\item \label{enu:0-1}For $\varepsilon$ sufficiently small, the solution $y_{\eps}$ can be extended to a solution
$y_{\eps}\in\Coo([0,1]_{\R},\R^{d})$ of \eqref{eq:geo_eq_gen} such
that $y_{\eps}(1)=q$.
\item \label{enu:C^2Conv}$y_{\varepsilon}\to u$ in $\mathcal{C}^{2}$.
\item \label{enu:y_eps-GSF}The net $(y_{\eps})$ defines a GSF, i.e.~$y:=[y_{\eps}(-)]\in\gsf_{>0}(p,q)$.
\end{enumerate}
\end{lem}
\begin{proof}
\textbf{~}\\
\textbf{Claim }\ref{enu:0-1}\textbf{, }\ref{enu:C^2Conv}: For all
$i$, $j$ , we have that $\Gamma_{ij}^{\varepsilon}\to\Gamma_{ij}$
locally uniformly. Thus, we obtain these claims by \eqref{eq:geo_eq_classical}
and by continuous dependence on parameters in ODE, see e.g.~\cite[Lemma 2.3]{KSS:14}.

\noindent \textbf{Claim} \ref{enu:y_eps-GSF} I: $y:=[y_{\eps}(-)]\in\gsf([0,1],\rc^{d})$\\
We have to show that for all $\varepsilon$ all derivatives of $y_{\varepsilon}$
are moderate. This is obviously true for $y_{\varepsilon},\dot{y}_{\varepsilon}$
and $\ddot{y}_{\varepsilon}$. The claim follows now from the fact
that 
\begin{align*}
\frac{\dd^{n+2}}{\dd t^{n+2}}y_{\varepsilon}=-\frac{\dd^{n}}{\dd t^{n}}\left(\Gamma_{ij}^{\varepsilon k}(y)\dot{y}_{\varepsilon}^{i}\dot{y}_{\varepsilon}^{j}\right)
\end{align*}
so that there exists a polynomial $P$ such that 
\begin{align*}
\frac{\dd^{n}}{\dd t^{n}}\left(\Gamma_{ij}^{\varepsilon k}(y)\dot{y}_{\varepsilon}^{i}\dot{y}_{\varepsilon}^{j}\right)=P\left(y_{\varepsilon},\frac{\dd}{\dd t}y_{\varepsilon},\ldots,\frac{\dd^{n+1}}{\dd t^{n+1}}y_{\varepsilon},\Gamma_{ij}^{\varepsilon k},\text{D}\Gamma_{ij}^{\varepsilon k},\ldots,\text{D}^n\Gamma_{ij}^{\varepsilon k}\right).
\end{align*}

\noindent \textbf{Claim} \ref{enu:y_eps-GSF} II: $|\dot{y}(t)|>0$
for all $t\in[0,1]$\\
By \ref{enu:C^2Conv}, we have that $y_{\varepsilon}\to u$ in
$\mathcal{C}^{2}$. Furthermore, $g_{\varepsilon}\to g$ in $\mathcal{C}^{1}$
by assumption, and we know that $g(\dot{u},\dot{u})=c>0$ for some
$c\in\R_{>0}$ since $u$ is a $g$-geodesic (cf.~\cite[Lemma 1.4.5]{J:11}).
Therefore, we obtain that $g_{\varepsilon}(\dot{y}_{\varepsilon},\dot{y}_{\varepsilon})>c/2>0$
for $\varepsilon>0$ small enough. 
\end{proof}
Finally, the standard part of the generalized length of $y$ is the
length of $u$:
\begin{thm}
Let $u$ and $y_{\varepsilon}$ be as above. We conclude (using Lemma
\ref{lem:conv_length}) that $\text{st}(L_{\tilde{g}}(y))=L_{g}(u)$. 
\end{thm}
\begin{prop}
Let $y=[y_{\varepsilon}(-)]$ be as above. In addition, assume that
each $y_{\varepsilon}$ is $L_{g_{\varepsilon}}$-minimizing. Then
$L_{\tilde{g}}(y)$ is minimal. 
\end{prop}
\begin{proof}
Let $\lambda=[\lambda_{\eps}(-)]\in\gsf_{>0}(p,q)$. We have that
$L_{\tilde{g}}(\lambda)=\left[L_{g_{\varepsilon}}(\lambda_{\varepsilon})\right]$
and that $L_{\tilde{g}}(y)=\left[L_{g_{\varepsilon}}(y_{\varepsilon})\right]$.
By assumption, for all $\eps$ we have 
\begin{align*}
L_{g_{\varepsilon}}(\lambda_{\varepsilon})\geq L_{g_{\varepsilon}}(y_{\varepsilon}).
\end{align*}
Therefore, $L_{\tilde{g}}(\lambda)\ge L_{\tilde{g}}(y)$, as claimed. 
\end{proof}
\begin{cor}
\label{cor:minimizer} Let $\lambda\in\gsf_{>0}(p,q)$ be a minimizer
of $L_{\tilde{g}}$ and assume that for $\eps$ small, $y_{\varepsilon}$
is $L_{g_{\varepsilon}}$-minimizing. Then $L_{\tilde{g}}(y)=L_{\tilde{g}}(\lambda)$. 
\end{cor}
This Corollary \ref{cor:minimizer} gives us a way to answer the question
if a certain classical geodesic between two given classical points $p$ and $q$ is a length-minimizer.

Furthermore, we are able to prove the following theorem, relating GSF-minimizers to classical minimizers.

\begin{thm}\label{thm:gsf_min_class_min}
  Let $p,q \in \R^d$ and let $\gamma\in \gsf_{>0}(p,q)$ such that $L_{\tilde g}(\gamma)$ is minimal. Assume that $\text{st}(L_{\tilde g}(\gamma))$ exists and that there exists $w\in\mathcal{C}_{>0}^{1}(p,q)$ such that $L_g(w)=\text{st}(L_{\tilde g}(\gamma))$.
  
  Then $w$ is $g$ - minimizing and a $g$ - geodesic.
\end{thm}

\begin{proof}
 Assume to the contrary that there exists a curve $\sigma\in C^2$ connecting $p$ and $q$ (w.\,l.\,o.\,g.\ $\sigma$ is a $g$ - geodesic) such that
 \begin{align*}
  L_g(\sigma) < L_g(w).
 \end{align*}
Now we construct (as done above) $g_\varepsilon$, $\sigma_\varepsilon$ and set $\tilde \sigma := [\sigma_\varepsilon]$. Then:
\begin{align*}
 \text{st}\left(L_{\tilde g}(\tilde \sigma)\right) = L_g(\sigma) < L_g(w) = \text{st}\left(L_{\tilde g}(\gamma)\right).
\end{align*}
But, by assumption we have that $L_{\tilde g}(\gamma) \leq L_{\tilde g}(\tilde \sigma)$, which implies 
\begin{align*}
\text{st}\left(L_{\tilde g}(\gamma)\right) \leq \text{st}\left(L_{\tilde g}(\tilde \sigma)\right) < \text{st}\left(L_{\tilde g}(\gamma)\right).
\end{align*}
This is a contradiction.
\end{proof}

\section{Conclusions}

We can summarize the present work as follows
\begin{enumerate}
\item The setting of GSF allows to treat Schwartz distributions more closely
to classical smooth functions. The framework is so flexible and the
extensions of classical results are so natural in many ways one may treat it like smooth functions.
\item One key step of the theory is the change of the ring of scalars into
a non-Archimedean one and the use of the strict order relation $<$
to deal with topological properties. So, the use of $<$ and of $\rc$-valued
norms allows a natural approach to topology, even of infinite dimensional
spaces (cf.~Def.~\ref{def:sharpTopSpaceGSF}). On the other hand,
the use of a ring with zero divisors and a non-total order relation
requires a more refined and careful analysis. However, as proved in
the present work, very frequently classical proofs can be formally
repeated in this context, but paying particular attention to using
the relation $<$, using invertibility instead of being non zero in
$\R$ and avoiding the total order property.
\item Others crucial properties are the closure of GSF with respect to composition
and the use of the gauge $\rho$, because they do not force to narrow
the theory into particular cases.
\item The present extension of the classical theory of calculus of variations
shows that GSF are a powerful analytical technique. The final application
shows how to use them as a method to address problems in an Archimedean
setting based on the real field $\R$.
\end{enumerate}
Concerning possible future developments, we can note that:
\begin{enumerate}[resume]
\item A generalization of the whole construction to piecewise GSF seems
possible.
\item A more elegant approach to integration of piecewise GSF could use
the existence of right and left limits of $(f_{1},\ldots,f_{n})(-)$
and hyperfinite Riemann-like sums, i.e.~sums
\[
\sum_{i=1}^{N}f(x'_{i})(x_{i}-x_{i-1}):=\left[\sum_{i=1}^{N_{\eps}}f_{\eps}(x'_{i,\eps})(x_{i,\eps}-x_{i-1,\eps})\right]\in\rc^{d}
\]
extended to $N\in\widetilde{\N}:=\left\{ \left[\text{int}(x_{\eps})\right]\mid[x_{\eps}]\in\rc\right\} $,
where $\text{int}(-)$ is the integer part function.
\end{enumerate}
The present work could lay the foundations for further works concerning
the possibility to extend other results of the calculus of variations
in this generalized setting.

\end{document}

%% file: Zeichnung_gsf.pdf_tex
\begingroup%
  \makeatletter%
  \providecommand\color[2][]{%
    \errmessage{(Inkscape) Color is used for the text in Inkscape, but the package 'color.sty' is not loaded}%
    \renewcommand\color[2][]{}%
  }%
  \providecommand\transparent[1]{%
    \errmessage{(Inkscape) Transparency is used (non-zero) for the text in Inkscape, but the package 'transparent.sty' is not loaded}%
    \renewcommand\transparent[1]{}%
  }%
  \providecommand\rotatebox[2]{#2}%
  \ifx\svgwidth\undefined%
    \setlength{\unitlength}{560.25458771bp}%
    \ifx\svgscale\undefined%
      \relax%
    \else%
      \setlength{\unitlength}{\unitlength * \real{\svgscale}}%
    \fi%
  \else%
    \setlength{\unitlength}{\svgwidth}%
  \fi%
  \global\let\svgwidth\undefined%
  \global\let\svgscale\undefined%
  \makeatother%
  \begin{picture}(1,0.66810023)%
    \put(0,0){\includegraphics[width=\unitlength]{Zeichnung_gsf.pdf}}%
    \put(0.47812837,0.06403766){\color[rgb]{0,0,0}\makebox(0,0)[lt]{\begin{minipage}{2.46396654\unitlength}\raggedright $t_{\varepsilon}$\end{minipage}}}%
    \put(0.03659711,0.04855634){\color[rgb]{0,0,0}\makebox(0,0)[lb]{\smash{$a$}}}%
    \put(0.86830999,0.04603697){\color[rgb]{0,0,0}\makebox(0,0)[lb]{\smash{$b$}}}%
    \put(0.48020702,0.57825381){\color[rgb]{0,0,0}\makebox(0,0)[lb]{\smash{$2\rho^{2k}_{\varepsilon}$}}}%
    \put(0.61672648,0.28426063){\color[rgb]{0,0,0}\makebox(0,0)[lb]{\smash{$2\rho^{2k}_{\varepsilon}$}}}%
    \put(0.26768023,0.04855634){\color[rgb]{0,0,0}\makebox(0,0)[lb]{\smash{$t_\varepsilon - \rho^k_\varepsilon$}}}%
    \put(0.64502522,0.04855634){\color[rgb]{0,0,0}\makebox(0,0)[lb]{\smash{$t_\varepsilon + \rho^k_\varepsilon$}}}%
    \put(0.33579891,0.45875107){\color[rgb]{0,0,0}\makebox(0,0)[lb]{\smash{$\lambda_\varepsilon \rho_\varepsilon^k$}}}%
    \put(0.22854675,0.27954309){\color[rgb]{0,0,0}\makebox(0,0)[lb]{\smash{$2\rho^{2k}_{\varepsilon}$}}}%
  \end{picture}%
\endgroup%

%% file: gsf_calc_var_final_2.bbl
\begin{thebibliography}{10}
\bibitem{Ar-Fe-Ju05}Aragona, J., Fernandez, R., Juriaans, S.O., A
discontinuous Colombeau differential calculus, Monatsh.~Math. \textbf{144},
13\textendash 29 (2005).

\bibitem{AJ:01}Aragona, J., Juriaans, S.O., Some structural properties of the topological ring of {C}olombeau's generalized numbers, Comm. Algebra \textbf{29} (2001).

\bibitem{Ave86}Avez, A., Differential Calculus, John Wiley \& Sons
Inc., 1986.

\bibitem{Col92}Colombeau, J.F., Multiplication of distributions -
A tool in mathematics, numerical engineering and theoretical Physics.
Springer-Verlag, Berlin Heidelberg (1992).

\bibitem{Dav88}Davie, A. M., Singular minimisers in the calculus
of variations in one dimension. Arch. Rational Mech. Anal. 101(2),
161\textendash 177, 1988.

\bibitem{Dir26}Dirac, P.A.M., The physical interpretation of the
quantum dynamics, \emph{Proc.\ R.\ Soc.\ Lond.\ A, }\textbf{\emph{113}}\emph{,}
1926\textendash 27, 621\textendash 641.

\bibitem{En-To-Ts05}Engquist, B., Tornberg, A.K., Tsai, R., Discretization
of Dirac delta functions in level set methods. Journal of Computational
Physics, 207:28\textendash 51, 2005.

\bibitem{GeFo00}Gelfand, I.M., Fomin, S.V., Calculus of variations,
Dover Publications, 2000.

\bibitem{GiKu16}Giordano, P., Kunzinger, M., Inverse Function Theorems
for Generalized Smooth Functions. Invited paper for the Special issue
ISAAC - Dedicated to Prof. Stevan Pilipovic for his 65 birthday. Eds.
M. Oberguggenberger, J. Toft, J. Vindas and P. Wahlberg, Springer
series \char`\"{}Operator Theory: Advances and Applications\char`\"{},
Birkhaeuser Basel, 2016.

\bibitem{GK15}Giordano, P., Kunzinger, M., A convenient notion of
compact sets for generalized functions. Accepted in Proceedings of
the Edinburgh Mathematical Society, 2016. See arXiv 1411.7292v1.

\bibitem{GiKu13}Giordano, P., Kunzinger, M., 'New topologies on Colombeau
generalized numbers and the Fermat-Reyes theorem'. Journal of Mathematical
Analysis and Applications 399 (2013) 229\textendash 238. 

\bibitem{Gi-Ku-St15}Giordano P., Kunzinger M., Steinbauer R., A new
approach to generalized functions for mathematical physics. See http://www.mat.univie.ac.at/\~{ }giordap7/GenFunMaps.pdf.

\bibitem{GKV}Giordano, P., Kunzinger, M., Vernaeve, H., Strongly
internal sets and generalized smooth functions. Journal of Mathematical
Analysis and Applications, volume 422, issue 1, 2015, pp. 56\textendash 71.

\bibitem{GiLu15}Giordano, P., Luperi Baglini, L., Asymptotic gauges:
Generalization of Colombeau type algebras. Math. Nachr.\ \textbf{289},
2-3, 1\textendash 28, (2015).

\bibitem{Gra30}Graves, L.M., Discontinuous solutions in the calculus
of variations. Bull. Amer. Math. Soc. 36, 831\textendash 846, 1930.

\bibitem{GKOS}Grosser, M., Kunzinger, M., Oberguggenberger, M., Steinbauer,
R., Geometric theory of generalized functions, Kluwer, Dordrecht (2001).

\bibitem{HE:73}Hawking, S.W., Ellis, G., The large scale structure of space-time, Cambridge University Press (1976)

\bibitem{Hi:76}Hirsch, M.W., Differential Topology, Springer (1976).

\bibitem{Ho-Ni-St16}Hosseini, B. , Nigam, N., Stockie, J.M., On regularizations
of the Dirac delta distribution, Journal of Computational Physics,
Volume 305, 2016, Pages 423\textendash 447.

\bibitem{J:11}Jost, J., Riemannian Geometry and Geometric Analysis,
Springer, 2011.

\bibitem{JoJo98}Jost, J., Li-Jost, X., Calculus of variations, Cambridge
Studies in Advanced Mathematics 64, 1998.

\bibitem{Kat-Tal12}Katz, M.G., Tall, D., A Cauchy-Dirac delta function.
\emph{Foundations of Science}, 2012. See http://dx.doi.org/10.1007/s10699-012-9289-4
and http://arxiv.org/abs/1206.0119.

\bibitem{KKO:08}Konjik, S., Kunzinger, M., Oberguggenberger, M.:
Foundations of the Calculus of Variations in Generalized Function
Algebras. Acta Applicandae Mathematicae \textbf{103} n.~2, 169\textendash 199
(2008)

\bibitem{KSSV:13}Kunzinger, M., Steinbauer, R., Stojkovi\'{c}, M.,
Vickers, J.A.,  A regularisation approach to causality theory for $\mathcal{C}^{1,1}$-Lorentzian
metrics, Gen. Relativ. Gravit. 46 (2014).

\bibitem{KSS:14}Kunzinger, M., Steinbauer, R., Stojkovi\'{c}, M.,
The exponential map of a ${C}^{1,1}$-metric, Diff. Geom. Appl. 34,
14 \textendash 24 (2014).

\bibitem{Lau89}Laugwitz, D., Definite values of infinite sums: aspects
of the foundations of infinitesimal analysis around 1820. Arch. Hist.
Exact Sci. \textbf{39} (1989), no. 3, 195\textendash 245.

\bibitem{LSS:13} Lecke, A., Steinbauer, R., Švarc, R., The regularity
of geodesics in impulsive pp-waves, Gen. Relativ. Gravit. 46 (2014).

\bibitem{PhDAlex}Lecke, A., Non-smooth Lorentzian Geometry and Causality
Theory, PhD Thesis, Universität Wien (2016).

\bibitem{Lu-Gi16}Luperi Baglini, L., Giordano, P., Fixed point iteration
methods for arbitrary generalized ODE, preprint.

\bibitem{Lu-Gi16b}Luperi Baglini, L., Giordan, P., The category of
Colombeau algebras. In revision for Monatshefte für Mathematik. See
arXiv 1507.02413.

\bibitem{LY:06}Lytchak, A., Yaman, A., On Hoelder continuous Riemannian
and Finsler metrics, Trans. Amer. Math. Soc. 358 (2006).

\bibitem{M:13}Minguzzi, E., Convex neighborhoods for Lipschitz connections
and sprays, Monatshefte für Mathemati, Volume 177, Issue 4, pp 569-625
(2015)

\bibitem{ObVe08}Oberguggenberger, M., Vernaeve, H., Internal sets
and internal functions in Colombeau theory, \emph{J.~Math.~Anal.~Appl.}~341
(2008) 649\textendash 659.

\bibitem{Rob73}Robinson, A., Function theory on some nonarchimedean
fields, Amer. Math. Monthly \textbf{80} (6) 87\textendash 109; Part
II: Papers in the Foundations of Mathematics (1973).

\bibitem{SSLP:16}Sämann, C., Steinbauer, R., Lecke, A., Podolský,
J.,Geodesics in nonexpanding impulsive gravitational waves with $\Lambda$,
part I, Classical and Quantum Gravity 33 (2016).

\bibitem{Sa-St15}Sämann,C. , Steinbauer, R., Geodesic Completeness
of Generalized Space-times. In Eds S. Pilipovi\'{c}, J., Pseudo-Differential
Operators and Generalized Functions, Volume 245 of the series Operator
Theory: Advances and Applications pp 243-253, 2015.

\bibitem{Milena}Stojkovi\'{c}, M., Causality theory for $\mathcal{C}^{1,1}$
- metrics, PhD Thesis, Universität Wien (2015).

\bibitem{To-En04}Tornberg, A.K., Engquist, B., Numerical approximations
of singular source terms in differential equations, Journal of Computational
Physics 200 (2004) 462\textendash 488.

\bibitem{Tuc93}Tuckey, C., Nonstandard methods in the calculus of
variations, Pitman Research Notes in Mathematics Series 297. Longman
Scientific \& Technical, Harlow, 1993.
\end{thebibliography}
